\documentclass[12pt]{article}

\usepackage{fullpage}
\usepackage[utf8]{inputenc}
\usepackage[T1]{fontenc}
\usepackage{graphicx}
\usepackage{amsmath,amsfonts,amssymb,mathtools}
\usepackage{stmaryrd}
\usepackage{graphicx}
\usepackage{mathrsfs,dsfont}
\usepackage{amsmath,amsthm,amsthm}

\usepackage{mathabx}

\usepackage{hyperref}

\usepackage{enumerate}

\usepackage{color}

\usepackage{authblk}

\newcommand{\E}{\mathbb{E}}
\newcommand{\R}{\mathbb{R}}
\newcommand{\N}{\mathbb{N}}

\newcommand{\T}{\mathbb{T}}

\newtheorem{theo}{Theorem}[section]

\newtheorem{rem}[theo]{Remark}

\newtheorem{propo}[theo]{Proposition}
\newtheorem{lemma}[theo]{Lemma}

\newtheorem{hyp}[theo]{Assumption}

\begin{document}

\title{Analysis of an Adaptive Biasing Force method based on self-interacting dynamics}

\author[,1]{Michel Bena\"im \thanks{\texttt{michel.benaim@unine.ch}}}
\author[,2]{Charles-Edouard Br\'ehier \thanks{\texttt{brehier@math.univ-lyon1.fr}}}
\author[,3]{Pierre Monmarch\'e \thanks{\texttt{pierre.monmarche@sorbonne-universite.fr}}}

\affil[1]{Institut de Math\'ematiques, Universit\'e de Neuch\^atel, Switzerland}
\affil[2]{Univ Lyon, CNRS, Universit\'e Claude Bernard Lyon 1, UMR5208, Institut Camille Jordan, F-69622 Villeurbanne, France}
\affil[3]{Sorbonne-Université, CNRS, Université de Paris, Laboratoire Jacques-Louis Lions (LJLL), F-75005 Paris, France}

\maketitle

\begin{abstract}
This article fills a gap in the mathematical analysis of Adaptive Biasing algorithms, which are extensively used in molecular dynamics computations. Given a reaction coordinate, ideally, the bias in the overdamped Langevin dynamics would be given by the gradient of the associated free energy function, which is unknown. We consider an adaptive biased version of the overdamped dynamics, where the bias depends on the past of the trajectory and is designed to approximate the free energy.

The main result of this article is the consistency and efficiency of this approach. More precisely we prove the almost sure convergence of the bias as time goes to infinity, and that the limit is close to the ideal bias, as an auxiliary parameter of the algorithm goes to $0$.

The proof is based on interpreting the process as a self-interacting dynamics, and on the study of a non-trivial fixed point problem for the limiting flow obtained using the ODE method.
\end{abstract}

\section{Introduction}

Let $\mu_\star$ be a probability distribution on the $d$-dimensional flat torus $\T^d$, of the type:
\begin{equation}\label{eq:mu}
d\mu_\star(x)=\frac{e^{-\beta V(x)}}{Z(\beta)}dx~,\quad Z(\beta)=\int_{\T^d}e^{-\beta V(x)}dx,
\end{equation}
where $dx$ is the normalized Lebesgue measure on $\T^d$. For applications in physics and chemistry (e.g. in molecular dynamics), $\mu_\star$ is referred to as the Boltzmann-Gibbs distribution associated with the potential energy function $V$ and the inverse temperature parameter $\beta>0$. For applications in statistics (e.g. in Bayesian statistics), $-\beta V$ is referred to as the log-likelihood. In this article, the function $V:\T^d\to\mathbb{R}$ is assumed to be smooth.

In order to estimate integrals of the type $\int \varphi d\mu_\star$, with $\varphi:\T^d\to\mathbb{R}$, probabilistic methods are used, especially when $d$ is large. The Markov Chain Monte Carlo (MCMC) method consists in interpreting the integral as the (almost sure) limit
\[
\int \varphi d\mu_\star=\underset{T\to\infty}\lim~\frac{1}{T}\int_{0}^{T}\varphi(X_t^0)dt=\underset{T\to\infty}\lim~\int \varphi d\mu_T^0,
\]
where $\mu_t^0=\frac{1}{T}\int_0^T\delta_{X_t^0}dt$ is the random empirical distribution associated with an ergodic Markov process $\bigl(X_t^0\bigr)_{t\ge 0}$, with unique invariant distribution $\mu_\star$. The choice of the Markov dynamics is not unique, and in this work we consider the overdamped Langevin dynamics
\[
dX_t^0=-\nabla V(X_t^0)dt+\sqrt{2\beta^{-1}}dW_t
\]
where $\bigl(W_t\bigr)_{t\ge 0}$ is a $d$-dimensional Wiener process. In practice, discrete-time Markov processes, defined for instance using the Metropolis-Hastings algorithm, are employed.

The convergence to equilibrium requires that the Markov process explores the entire energy landscape, which may be a very slow process. Indeed, in practical problems, the dimension $d$, {\it i.e.} the number of degrees of freedom in the system, is very large, and the probability distribution $\mu_\star$ is multimodal: the function $V$ admits several local minima (interpreted as potential energy wells) and $\beta$ is large. In that situation, the Markov process is metastable: when it reaches an energy well, it tends to stay there for a long time (whose expectation goes to infinity when $\beta$ goes to infinity) before hopping to another energy well. Asymptotic results for the exit time from energy wells when $\beta\to\infty$ are given by Eyring-Kramers type formulas \cite{BdH_book,LelievreMetastable}. The metastability of the process substantially slows down the exploration of the energy landscape, hence the convergence when $T\to\infty$ towards the target quantity $\int \varphi d\mu_\star$.

In the development of Monte-Carlo methods in the last decades, many techniques have been studied in order to efficiently sample multimodal distributions. The bottom-line strategy to enhance sampling consists in biasing the dynamics and in reweighting the averages: indeed, for any smooth function $\tilde{V}:\T^d\to\mathbb{R}$, one has
\[
\int \varphi d\mu_\star=\frac{\int \varphi e^{-\beta V}}{\int e^{-\beta V}}=\frac{\int\varphi e^{-\beta(V-\tilde{V})}~e^{-\beta\tilde{V}}}{\int e^{-\beta(V-\tilde{V})}~e^{-\beta\tilde{V}}}=\underset{t\to\infty}\lim~\frac{\int_0^t \varphi(\tilde{X}_s)e^{-\beta(V(\tilde{X}_s)-\tilde{V}(\tilde{X}_s)}ds}{\int_0^t e^{-\beta(V(\tilde{X}_s)-\tilde{V}(\tilde{X}_s)}ds},
\]
where the biased dynamics is given by $d\tilde{X}_t=-\nabla \tilde{V}(\tilde{X}_t)dt+\sqrt{2\beta^{-1}}dW_t$. This is nothing but an Importance Sampling method, and choosing carefully the function $\tilde{V}$ may substantially reduce the computational cost. Indeed, if the distribution with density proportional to $e^{-\beta\tilde{V}(x)}$ is not multimodal, the biased process $\tilde{X}_t$ converges to equilibrium and explores the state space faster than the unbiased process $X_t$. In the sequel, we explain how to choose $\tilde{V}$ in order to benefit from the importance sampling strategy.

From now on, in order to simplify the notation, $\beta=1$. In addition, without loss of generality, assume that $\int_{\T^d} e^{-V(x)}dx=1$.

Instead of treating the problem in an intractable full generality, we focus on the typical situation when some additional a priori knowledge on the system is available. Precisely, let $\xi:\T^d\to \T^m$ be a smooth function, which is referred to as the reaction coordinate (following the terminology employed in the molecular dynamics community). Let us stress that the identification of appropriate reaction coordinates is a delicate question, which depends on the system at hand. The problem of automatic learning   of good reaction coordinates currently generates a lot of research, see for instance \cite{Schutte,Brandt} and references within. We do not consider this question in the sequel.

The biasing potential in the importance sampling schemes considered in this work will be of the type $\tilde{V}(x)=V(x)-A(\xi(x))$, where $A:\T^m\to\mathbb{R}$. In practice, the number of macroscopic variables $m$ is very small compared to the dimension $d$ of the model (which describes the full microscopic system). As will be explained below, without loss of generality, we assume that $\xi(x)=\xi(y,z)=z$ for all $x=(y,z)\in\T^{d-m}\times\T^m$. This expression for the reaction coordinate simplifies the presentation of the method, however considering more general reaction coordinates $\xi$ is possible up to adapting some definitions below. To explain the construction of the method and to justify its efficiency, we assume that the reaction coordinate is representative of the metastable behavior of the system: roughly, this means that only the exploration in the $z$ variable is affected by the metastability, whereas the exploration in the $y$ variable is much faster.

In this framework, the fundamental object is the free energy function $A_\star$ defined as follows: for all $z\in\T^m$,
\begin{equation}\label{eq:FE}
A_\star(z)=-\log\Bigl(\int_{\T^{d-m}}e^{-V(y,z)}dy\Bigr).
\end{equation}
For general considerations on the free energy and related computational aspects, we refer to \cite{LRS_book,LS_acta}. By construction, if $X=(Y,Z)$ is a random variable with distribution $\mu_\star$, then the marginal distribution of $Z$ is given by
\[
d\nu_\star(z)=e^{-A_\star(z)}dz.
\]
Introduce the notation $(Y_t^0,Z_t^0)=X_t^0$ for the solution of the overdamped Langevin dynamics
\[
\begin{cases}
dY_t^0=-\nabla_yV(Y_t^0,Z_t^0)dt+\sqrt{2}dW_t^{(d-m)},\\
dZ_t^0=-\nabla_zV(Y_t^0,Z_t^0)dt+\sqrt{2}dW_t^{(m)},
\end{cases}
\]
where $W_t=(W_t^{(d-m)},W_t^{(m)})$. It $\nu_t^0=\frac{1}{t}\int_0^t \delta_{Z_s^0}ds$ denotes the empirical distribution for the variable $Z^0$, then almost surely
\[
\nu_t^0\underset{t\to\infty}\to \nu_\star,
\]
in the sense of weak convergence in the set $\mathcal{P}(\T^m)$ of probability distributions on $\T^m$. Since the reaction coordinate is representative of the metastability of the system, this convergence shares the same computational issues as when considering the full process $X^0$.

A much better performance can be attained considering the following biased dynamics, where $V(x)$ is replaced by $\tilde{V}_\star(x)=V(x)-A_\star(\xi(x))$:
\[
\begin{cases}
dY_t^\star=-\nabla_yV(Y_t^\star,Z_t^\star)dt+\sqrt{2}dW_t^{(d-m)},\\
dZ_t^\star=-\nabla_zV(Y_t^\star,Z_t^\star)dt+\nabla A_\star(Z_t^\star)dt+\sqrt{2}dW_t^{(m)}.
\end{cases}
\]
Define the associated empirical measures on $\T^d$ and $\T^m$ respectively:
\[
\mu_t^\star=\frac{1}{t}\int_0^t\delta_{X_s^\star}ds~,\quad \nu_t^\star=\frac{1}{t}\int_0^t\delta_{Z_s^\star}ds,
\]
where $X_s^\star=(Y_s^\star,Z_s^\star)$. As explained above, $\int\varphi d\mu_\star$ can then be computed by the reweighting procedure. Observe that by ergodicity for $\bigl(X_t^\star\bigr)_{t\ge 0}$ and the definition of $A_\star$, one has
\[
\nu_t^\star\underset{t\to\infty}\to dz,
\]
{\it i.e.} at the limit the distribution of $Z_t^\star$ is uniform on $\T^m$. This observation, which is referred to as the {\it flat histogram property} in the literature devoted to applications, means that the process $X^\star$ does not suffer from slow convergence to equilibrium due to energy barriers, compared to the process $X^0$.

In practive, the free energy function $A_\star$ is not known, thus the ideal approach described above is not applicable. In fact, in many applications, the real objective is the computation of the free energy function. One of the important features of many free energy computation algorithms, such as the one studied in this work, is to compute an approximation of the free energy function on-the-fly, and to use this approximation to enhance sampling. Checking that such adaptive algorithms are efficient and consistent requires careful mathematical analysis.

In this article, we consider a class of adaptive biasing methods, where the dynamics is of the form
\begin{equation}\label{eq:ABF-intro}
\begin{cases}
dY_t=-\nabla_yV(Y_t,Z_t)dt+\sqrt{2}dW_t^{(d-m)},\\
dZ_t=-\nabla_zV(Y_t,Z_t)dt+\nabla A_t(Z_t)dt+\sqrt{2}dW_t^{(m)},
\end{cases}
\end{equation}
where the function $A_t$ depends on time $t$, approximates $A_\star$ when $t\to\infty$, and is defined in terms of the empirical measure
\begin{equation}\label{eq:mut-intro}
\mu_t=\frac{1}{t}\int_{0}^{t}\delta_{X_s}ds.
\end{equation}
The process $\bigl(X_t\bigr)_{t\ge 0}=\bigl(Y_t,Z_t\bigr)_{t\ge 0}$ is not a Markov process, instead it is a self-interacting diffusion process. The precise construction of the algorithm studied in this article is provided below.

This article is organized as follows. The construction of the algorithm~\eqref{eq:ABF} studied in this work is presented in Section~\ref{sec:ABF} below. The main result, Theorem~\ref{theo:Main}, is stated in Section~\ref{sec:main+discussion}, and a comparison with the literature is given. Section~\ref{sec:proof-wellposed} gives a proof of the well-posedness of the self-interacting dynamics~\eqref{eq:ABF} (Proposition~\ref{propo:wellposed}). Section~\ref{sec:ODE} exhibits the limiting flow (obtained by applying the ODE method) and establishes the asymptotic pseudotrajectory property. Finally, Section~\ref{sec:proof-main} provides the final crucial ingredients for the proof of the main result, Theorem~\ref{theo:Main}: a PDE estimate which provides some uniform bounds, and a global asymptotic stability property for the limiting flow.

\section{The Adaptive Biasing Force algorithm}\label{sec:ABF}

The objectives of this section are to define the Adaptive Biasing Force method \cite{ABF14} studied in this article, and to state the main results.

Recall the definitions~\eqref{eq:mu} and~\eqref{eq:FE} of the target distribution $\mu_\star$ and of the free energy $A_\star$ respectively. The potential energy function $V$ is assumed to be of class $\mathcal{C}^\infty$.

The reaction coordinate $\xi:\T^d\to\T^m$ satisfies $\xi(y,z)=z$ for all $x=(y,z)\in\T^d$. This expression substantially simplifies the presentation compared with a more general choice of $\xi:\T^d\to\R^m$. In applications, this is not restrictive, and consists in considering the so-called extended ABF algorithm \cite{ChipotEABF}. Precisely, an auxiliary variable $\mathcal{Z}$ is added to the state space, the extended potential energy function for $\overline{X}=(X,\mathcal{Z})$ is given by $\overline{V}(\overline{X})=V(X)+\frac{1}{2\sigma^2}|\xi(X)-\mathcal{Z}|^2$, where $\sigma>0$ is a small parameter, and one sets $\overline{\xi}(\overline{X})=\mathcal{Z}$.

\subsection{Construction}
The definition of the algorithm requires to make precise how in the evolution equation~\eqref{eq:ABF-intro}, the biasing potential function $A_t$, or its gradient $\nabla A_t$, is determined in terms of the empirical distribution $\mu_t$ given by~\eqref{eq:mut-intro}. The algorithm is based on the following identity: the gradient $\nabla A_\star$ of the free energy function $A_\star$ defined by~\eqref{eq:FE} is given by
\begin{equation}\label{eq:meanforce}
\nabla A_\star(z)=\frac{\int_{\T^{d-m}}\nabla_zV(y,z)e^{-V(y,z)}dy}{\int_{\T^{d-m}}e^{-V(y,z)}dy}=\E_{\mu_\star}\bigl[\nabla_zV(Y,Z)~\big|~Z=z\bigr].
\end{equation}
More generally, let $A:\T^m\to\mathbb{R}$ be a smooth function, and let $d\mu_\star^A(x)\propto e^{A(z)}d\mu_\star(y,z)$ be the ergodic invariant distribution of
\[
\begin{cases}
dY_t^A=-\nabla_yV(Y_t^A,Z_t^A)dt+\sqrt{2}dW_t^{(d-m)},\\
dZ_t^A=-\nabla_zV(Y_t^A,Z_t^A)dt+\nabla A(Z_t^A)dt+\sqrt{2}dW_t^{(m)}.
\end{cases}
\]
Then one has the identity 
\begin{eqnarray}\label{eq:nablaAstar}
\nabla A_\star(z)& =&\E_{\mu_\star^A}\bigl[\nabla_zV(Y,Z)~\big|~Z=z\bigr] .
\end{eqnarray}
The expressions for the gradient of the free energy function in equations~\eqref{eq:meanforce} and~\eqref{eq:nablaAstar} are simpler than (for instance) the expressions (5) and (6) in~\cite{LRS08} which hold for a general reaction coordinate mapping $\xi$, whereas we consider only the case $\xi(y,z)=z$.

The occupation measures $\mu_t$ defined by~\eqref{eq:mut-intro} are in general singular with respect to the Lebesgue measure on $\T^m$. In order to define the mapping $\mu_t\mapsto A_t$, we introduce a regularization kernel $K_\epsilon$, depending on the parameter $\epsilon\in(0,1]$, such that
\[
\nabla A_\star(z)=\underset{\epsilon\to 0}\lim~\frac {\iint_{\T^d}\nabla_zV(y,z')K_\epsilon(z',z)d\mu_\star(y,z')}{\iint_{\T^d}K_\epsilon(z',z)d\mu_\star(y,z')}.
\]
Indeed, formally, the expression~\eqref{eq:meanforce} for $\nabla A_\star$ is obtained with the kernel $K_\epsilon(z,z')$ replaced by a Dirac distribution $\delta(z-z')$. See Assumption~\ref{ass:kernel} below for precise conditions on the kernel function $K_\epsilon$.

For every $\epsilon\in(0,1]$ and $\mu\in\mathcal{P}(\T^d)$, define the mapping $F^\epsilon[\mu]:\T^m\to\mathbb{R}^m$ as follows:
\begin{equation}\label{eq:def-Fepsilon}
F^\epsilon[\mu](\cdot)=\frac{\iint \nabla_zV(y,z)K_\epsilon(z,\cdot)d\mu(y,z)}{\iint K_\epsilon(z,\cdot)d\mu(y,z)}.
\end{equation}
Due to the action of the regularization kernel $K_\epsilon$, in general $F^\epsilon[\mu]$ cannot be written as a gradient. For instance if $m=1$, a smooth function $F:\T\to\mathbb{R}$ is a gradient if and only if its average value is zero $\int F(z)dz=0$; in general, $\int F^\epsilon[\mu](z)dz\neq 0$.

The last ingredient in the construction is a projection operator $\mathbf{P}$, such that one defines $\nabla A^\epsilon[\mu]=\mathbf{P}(F^\epsilon[\mu])$. More precisely, for every $\epsilon\in(0,1]$ and $\mu\in\mathcal{P}(\T^d)$, define the mapping  $A^\epsilon[\mu]$ as follows:
\begin{equation}\label{eq:def-Aepsilon}
\begin{aligned}
A^\epsilon[\mu]&=\underset{A\in H^1(\T^m),\int A(z)dz=0}{\rm argmin}~\int\big|F^\epsilon[\mu](z)-\nabla A(z)\big|^2dz\,.
\end{aligned}
\end{equation}
As will be explained below, $A^\epsilon[\mu]$ is solution of an elliptic PDE. Note that $F^{\epsilon}[\mu]$ and $A^\epsilon[\mu]$ are functions depending only on $z\in\T^m$, with a dimension $m$ much smaller than $d$ the total number of degrees of freedom of the system. Typically, one has $m\in\{1,2,3\}$, which makes it possible to use the algorithm in practice.

\bigskip

We are now in position to define the process considered in this article: it is the solution of the system
\begin{equation}\label{eq:ABF}
\begin{cases}
dY_t=-\nabla_yV(Y_t,Z_t)dt+\sqrt{2}dW_t^{(d-m)},\\
dZ_t=-\nabla_zV(Y_t,Z_t)dt+\nabla A_t(Z_t)dt+\sqrt{2}dW_t^{(m)},\\
A_t=A^\epsilon[\mu_t],\\
\mu_t=\frac{1}{t}\int_0^t\delta_{(Y_s,Z_s)}ds.
\end{cases}
\end{equation}
Arbitrary (deterministic) initial conditions $Y_0=y_0\in\T^{d-m}$, $Z_0=z_0\in\T^m$, $\mu_0=\delta_{(y_0,z_0)}$ and $A_0=A^\epsilon[\mu_0]$ are provided. This process belongs to the class of self-interacting diffusions, see  \cite{BLR02,BR03,BR05,BR11} for standard references.

\subsection{Well-posedness of the system~\eqref{eq:ABF}}

Recall that $V:\T^d\to\R$ is assumed to be of class $\mathcal{C}^\infty$. Let us first state the assumptions satisfied by the kernel function $K_\epsilon$.
\begin{hyp}\label{ass:kernel}
For any $\epsilon\in(0,1]$, the mapping $K_\epsilon:\T^m\times\T^m\to (0,\infty)$ is of class $\mathcal{C}^\infty$ and positive.

For all $z\in\T^m$, one has
\[\int K_\epsilon(z,\cdot)dz=\int K_\epsilon(\cdot,z)dz=1\]
In addition, if $\psi:\T^d\to\R$ is a continuous and bounded function, one has
\[
\iint_{\T^d}\psi(y,z')K_\epsilon(z',z)dydz'\underset{\epsilon\to 0}\to \int_{\T^{d-m}}\psi(y,z)dy~,\quad \forall~z\in\T^m.
\]
Finally, there exists $c_K\in(0,\infty)$, such that
\[
\underset{z\in\T^m}\sup~\int_{\T^m}|z-z'|^2\bigl(K_\epsilon(z',z)+K_\epsilon(z,z')\bigr)dz'\le c_K\epsilon.
\]
\end{hyp}
Define $m_\epsilon=\underset{z,z'\in\T^m}\min~K_\epsilon(z',z)$ and $M_\epsilon^{(k)}=\underset{z,z'\in\T^m}\max~|\nabla_z^k K_\epsilon(z',z)|+\underset{z,z'\in\T^m}\max~|\nabla_{z'}^k K_\epsilon(z',z)|$, where $k$ is a nonnegative integer and $\nabla^k$ denotes the derivative of order $k$. Owing to Assumption~\ref{ass:kernel}, one has $m_\epsilon>0$ and $M_\epsilon^{(k)}<\infty$ for all $\epsilon\in(0,1]$, however these estimates are not uniform with respect to $\epsilon$, {\it i.e.} $\underset{\epsilon\in(0,1]}\inf~m_\epsilon=0$ and $\underset{\epsilon\in(0,1]}\sup~M_\epsilon^{(k)}=\infty$.

Note that to establish the well-posedness of the system~\eqref{eq:ABF}, where $\epsilon\in(0,1]$ is fixed, upper bounds are allowed to depend on $\epsilon$. However, it will be crucial in Section~\ref{sec:proof-main} to derive some upper bounds which are uniform with respect to $\epsilon$ in order to prove the convergence when $t$ goes to infinity of $\mu_t$ and $A_t$ (to a limit depending on $\epsilon$), see Proposition~\ref{propo:bound}.

The exact form of the kernel function $K_\epsilon$ has no influence on the analysis below. Let us give an example: let $K_\epsilon(z^1,z^2)=\prod_{j=1}^{m}k_\epsilon\bigl(z_j^2-z_j^1\bigr)$, where for all $z\in\T$,
\[
k_\epsilon(z)=Z_\epsilon^{-1}\exp\bigl(-\frac{\sin^2(z/2)}{\epsilon^2/2}\bigr)
\]
is the so-called von-Mises kernel.

Owing to Assumption~\ref{ass:kernel}, it is straightforward to check that $F^\epsilon[\mu]$ is of class $\mathcal{C}^\infty$, for any $\mu\in\mathcal{P}(\T^d)$. Then the mapping ${A}^\epsilon[\mu]$ is the solution of the elliptic linear partial differential equation
\[
\Delta {A}^\epsilon[\mu]={\rm div}(F^\epsilon[\mu])
\]
and standard elliptic regularity theory implies that ${A}^\epsilon[\mu]$ is also of class $\mathcal{C}^\infty$. See Lemma~\ref{lem:bounds-F_A} below for quantitative bounds (depending on $\epsilon$).

\begin{propo}\label{propo:wellposed}
Under Assumption \ref{ass:kernel}, for any initial conditions $x_0=(y_0,z_0)\in\T^d$, the system~\eqref{eq:ABF} admits a unique solution, which is defined for all times $t\ge 0$.
\end{propo}

The proof of Proposition~\ref{propo:wellposed} is postponed to Section~\ref{sec:proof-wellposed}

\subsection{Main result and discussion}\label{sec:main+discussion}

Remark that the free energy can be defined up to an additive constant. Above,  $A_\star$ has been normalized so that $\int_{\T^m} e^{-A_\star} d z = 1$, while $A_t$ is such that $\int_{\T^m} A_t dz=0$. Denote $\bar A_\star = A_\star - \int_{\T^m} A_\star(z) dz$. The standard norm on the Sobolev space $W^{1,p}(\T^m)$, for $p\in[2,\infty)$, is denoted by $\|\cdot\|_{W^{1,p}}$.

\begin{theo}\label{theo:Main}
Under Assumption \ref{ass:kernel}, there exists $\epsilon_0>0$ and, for all $p\in[1,+\infty)$, there exists $C_p\in[0,+\infty)$  such that, for all $\epsilon\in(0,\epsilon_0]$, there exists a unique probability distribution $\mu_\infty^\epsilon \in\mathcal P(\T^d)$ which satisfies
\[
d\mu_\infty^{\epsilon}(x)=d\mu_{\star}^{A^\epsilon[\mu_\infty^\epsilon]}(x)\propto e^{A^\epsilon[\mu_\infty^\epsilon](z)}d\mu_\star(y,z).
\]
In addition, one has the error estimate
\begin{eqnarray*}
\|\bar A_\star - A^\epsilon[\mu_\infty^\epsilon]\|_{W^{1,p}} & \leqslant & C_p\sqrt{\epsilon}\,,
\end{eqnarray*}
and, for any initial conditions $x_0=(y_0,z_0)\in\T^d$, almost surely, one has the convergence
\begin{eqnarray*}
\|A_t - A^\epsilon[\mu_\infty^\epsilon]\|_{W^{1,p}} &  \underset{t\rightarrow \infty}\longrightarrow & 0\\
\mu_t & \underset{t\rightarrow \infty}\longrightarrow & \mu_\infty^\epsilon\,,
\end{eqnarray*}
the latter in the sense of weak convergence in the set $\mathcal{P}(\T^d)$.
\end{theo}

The first identity in Theorem~\ref{theo:Main} means that the limit $\mu_\infty^\epsilon$ of $\mu_t$ is the fixed point of the mapping $\mu\mapsto \mu_\star^{A^\epsilon[\mu]}$, see Section~\ref{sec:ODE}. Equivalently, the limit $A^\epsilon[\mu_\infty^\epsilon]$ of $A_t$ is the fixed point of the mapping $A\mapsto A^\epsilon[\mu_\star^A]$, where we recall that $d\mu_\star^A(x)=e^{A(z)}d\mu_\star(y,z)$.

The almost sure convergence results of Theorem~\ref{theo:Main} may be loosely rephrased as follows
\[
\underset{\epsilon\to 0}\lim~\underset{t\to\infty}\lim~A_t=A_\star~,\quad \underset{\epsilon\to 0}\lim~\underset{t\to\infty}\lim~\mu_t=\mu_\star^{A_\star},
\]
and implies that the empirical distribution $\nu_t=\frac{1}{t}\int_{0}^{t}\delta_{\xi(X_s)}ds$ satisfies the approximate asymptotic flat-histogram property
\[
\underset{\epsilon\to 0}\lim~\underset{t\to\infty}\lim~\nu_t=dz.
\]
We stress that $\mu_\infty^\epsilon$ is not close (when $\epsilon\to 0$) to the multimodal target distribution $\mu_\star$: with the notation above one has $\mu_\star=\mu_\star^0\neq \mu_\star^{A_\star}$. However, the algorithm gives a way to approximate $\int\varphi d\mu_\star$ by reweighting: using the Cesaro Lemma, it is straightforward to check that one has
\[
\underset{t\to\infty}\lim~\frac{\int_0^t\varphi(X_s)e^{-A_s(Z_s)}ds}{\int_0^te^{-A_s(Z_s)}ds}=\underset{t\to\infty}\lim~\frac{\int_0^t\varphi(X_s)e^{-A^\epsilon[\mu_\infty^\epsilon](Z_s)}ds}{\int_0^te^{-A^\epsilon[\mu_\infty^\epsilon](Z_s)}ds}=\frac{\int \varphi(y,z) e^{-A^\epsilon[\mu_\infty^\epsilon](z)}d\mu_\infty^\epsilon(y,z)}{\int e^{-A^\epsilon[\mu_\infty^\epsilon](z)}d\mu_\infty^\epsilon(y,z)}=\int \varphi d\mu_\star,
\]
for any smooth $\varphi:\T^d\to\R$. Indeed, by the Sobolev embedding $W^{1,p}(\T^m)\subset \mathcal{C}^0(\T^m)$ if $p>m$, $A_t$ converges to $A^\epsilon[\mu_\infty^\epsilon]$ uniformly on $\T^m$.

Up to an error depending only on the width $\epsilon>0$ of the kernel function $K_\epsilon$, the adaptive algorithm~\eqref{eq:ABF} is thus a consistent way to approximately compute $\int\varphi d\mu_\star$, as well as the free energy function $A_\star$. The approximate asymptotic flat-histogram property stated above shows that the sampling in the slow, macroscopic variable $z$ is enhanced, hence the efficiency of the approach. Such results are a mathematical justification for the use of the ABF method based on self-interating dynamics in practical computations.

\begin{rem}
From Theorem \ref{theo:Main}, we expect the following Central Limit Theorem to hold: for all bounded $\varphi$ on $\T^d$,
\[\sqrt t \left( \int \varphi d\mu_t - \int \varphi d\mu_\infty^\epsilon\right) \ \overset{law}{\underset{t\rightarrow\infty}\longrightarrow} \ \mathcal N(0,\sigma_\varphi)\]
where $\sigma_\varphi$ is the asymptotical variance obtained by considering the process with a constant bias $\nabla A^\epsilon[\mu_\infty^\epsilon]$. Nevertheless, the proof of such a result, extending \cite[Theorem 4.III.5]{Duflo96}  at the cost of technical considerations, exceeds the scope of the present article.
\end{rem}

\begin{rem}
The convergence of $A_t$ to $A^\epsilon[\mu_\infty^\epsilon]$ when $t\to \infty$ in fact holds for $\mathcal{C}^k$ norms, for all integers $k$. However, the convergence of $\bar A_\star - A^\epsilon[\mu_\infty^\epsilon]$ when $\epsilon\to 0$ can be obtained only in $W^{1,p}$, for all $p\in[2,\infty)$ (hence in $\mathcal{C}^0$ due to a Sobolev embedding, for $p>m$). In fact, higher-order derivatives of $F^\epsilon[\mu]$ (and of $A^\epsilon[\mu]$) are expected to explode when $\epsilon\to 0$.
\end{rem}

The ABF has originally been introduced in \cite{ABF_1} in the molecular dynamics community, where it is widely used, see \cite{ABF_2,D17,ABF14}. An example of application in statistics is developed in \cite{CLS:12}. Another popular related biasing algorithm is the metadynamics algorithm \cite{M},\cite{WTM},\cite{JLZ:19},\cite{BCG:15}.

From a theoretical point of view, several variants of the ABF algorithm have been considered in various works. In a series of papers \cite{LRS08,AlrachidLelievre:15,LRS07,LM11}, Leli\`evre and his co-authors  considered a process similar to \eqref{eq:ABF} except that $\mu_t$ is replaced by the law of $X_t$. This corresponds to the mean-field limit of a system of $N$ interacting particles as $N$ goes to infinity \cite{JLR10}. The law of $X_t$ then solves a non-linear PDE, and long-time convergence is established through entropy techniques. In practice in fact, the bias $A_t$ is obtained both from interacting particles and from interaction with the past trajectories, so that $\mu_t$ is the empirical distribution of a system of $N$ replicas of the system $(X_t,Y_t)$ that contributes all to the same bias $A_t$.

The case of adaptive bias algorithm with a self-interacting process is addressed in \cite{EhrlacherLelievreMonmarche2019} for the ABF algorithm and in \cite{BenaimBrehier:16,BenaimBrehier:19} for the related adaptive biasing potential (ABP) algorithm. We emphasize on the fact that in these works, $\mu_t$ is replaced by a weighted empirical measure $\bar \mu_t$ given, in the spirit of an importance sampling scheme, by 
\[\bar \mu_t \ = \ \left(\int_0^t e^{-A_s(Z_s)} ds \right)^{-1}\int_0^t \delta_{X_s}e^{-A_s(Z_s)} ds\,.\]
Contrary to $\mu_t$ in Theorem~\ref{theo:Main}, this weighted empirical measure converges toward $\mu_\star$. This makes the theoretical study simpler than in the present case. However, in practice, there should be no reason to use this weighting procedure for ABF due to the identity~\eqref{eq:nablaAstar}. Indeed, provided that $A_t$ converges to some $A_\infty$, in the idealized case where $K_\epsilon$ is a Dirac mass, then  \eqref{eq:nablaAstar} implies that necessarily $A_\infty = A_\star$. This is no more true as soon as $\epsilon>0$ (which is necessary for the well-posedness of the algorithm), and one of the main motivation of the present work was to determine whether the convergence of  the natural (non re-weighted) version of ABF, which is the one used in practice, was robust with respect to the regularization step. Our results shows that this is true, provided $\epsilon$ is small enough.

\subsection{Notation}

Let $\N=\left\{1,\ldots\right\}$ and $\N_0=\N\cup\{0\}$, and let $k\in\N_0$ be a nonnegative integer. Let $\mathcal{C}^k(\T^{n_1},\R^{n_2})$ be the space of functions of class $\mathcal{C}^k$ on $\T^{n_1}$ with values on $\R^{n_2}$. The derivative of order $k$ is denoted by $\nabla^k$. The space $\mathcal{C}^k(\T^{n_1},\R^{n_2})$ is equipped with the norm $\|\cdot\|_{\mathcal{C}^k}$, defined by
\[
\|\phi\|_{\mathcal{C}^k}=\sum_{\ell=0}^{k}\|\nabla^k \phi\|_{\mathcal{C}^0},
\]
with $\|\phi\|_{\mathcal{C}^0}=\underset{z\in \T^{n_1}}\max~\|\phi(x)\|$. To simplify, the dimensions $n_1$ and $n_2$ are omitted in the notation for the norm $\|\cdot\|_{\mathcal{C}^k}$.

If $\phi:\T^{n_1}\to\R^{n_2}$ is a Lipschitz continuous function, its Lipschitz constant is denoted by ${\rm Lip}(\phi)$.

The space $\mathcal{P}(\T^d)$ of probability distributions on $\T^d$ (equipped with the Borel $\sigma$-field) is equipped with the total variation distance $d_{\rm TV}$ and with the Wasserstein distance $d_{\mathcal{W}_1}$. Recall that one has the following characterizations:
\[
\begin{aligned}
d_{TV}(\mu_1,\mu_2)=\underset{\psi:\T^d\to\R, \|\psi\|_\infty\le 1}\sup~\frac12\big|\int \psi d\mu_2-\int \psi d\mu_1\big|,\\
d_{\mathcal{W}_1}(\mu_1,\mu_2)=\underset{\psi:\T^d\to\R, {\rm Lip}(\psi)\le 1}\sup~\big|\int \psi d\mu_2-\int \psi d\mu_1\big|
\end{aligned}
\]
where for the total variation distance the supremum is taken over bounded measurable functions $\psi$.

The space $\mathcal{P}(\T^d)$ is also equipped with the following distance, which generates the topology of weak convergence:
\[
d_w(\mu_1,\mu_2)=\sum_{n\in\N}\frac{1}{2^n}\frac{\big|\int f_nd\mu_2-\int f_n d\mu_1\big|}{1+\big|\int f_nd\mu_2-\int f_n d\mu_1\big|},
\]
where the sequence $\mathcal{S}=\{f_n\}_{n\in\N}$ is dense in $\mathcal{C}^0(\T^d,\R)$, and, for all $n\in\N$, one has $f_n\in\mathcal{C}^\infty$ and $\|f_n\|_{\mathcal{C}^0}\le 1$.

\section{Proof of the well-posedness result Proposition~\ref{propo:wellposed}}\label{sec:proof-wellposed}

The objective of this section is to prove Proposition~\ref{propo:wellposed}, which states that the system~\eqref{eq:ABF} is well-posed. Some auxiliary estimates are provided, where the upper bounds are allowed to depend on the parameter $\epsilon$. Lemma~\ref{lem:bounds-F_A} provides estimates for $F^\epsilon[\mu]$ and $A^\epsilon[\mu]$, in $\mathcal{C}^k$, uniformly with respect to $\mu$. Lemma~\ref{lem:Lip-F_A} provides some Lipschitz continuity estimates with respect to $\mu$, in total variation and Wasserstein distances.

\subsection{Auxiliary estimates}

\begin{lemma}\label{lem:bounds-F_A}
For all $\epsilon\in (0,1]$ and $k\in\N_0$, there exists $C_{\epsilon,k}\in(0,\infty)$ such that one has
\[
\underset{\mu\in\mathcal{P}(\T^d)}\sup~\Bigl(\|F^\epsilon[\mu]\|_{\mathcal{C}^k(\T^m,\R^m)}+\|A^\epsilon[\mu]\|_{\mathcal{C}^k(\T^m,\R)}\Bigr)\le C_{\epsilon,k}.
\]
\end{lemma}

\begin{proof}[Proof of Lemma~\ref{lem:bounds-F_A}]
Observe that
\[
F^\epsilon[\mu]=\frac{F_{\rm aux}[\mu,\nabla_zV]}{F_{\rm aux}[\mu,1]},
\]
where $F_{\rm aux}^\epsilon[\mu,\psi]=\iint \psi(y,z)K_\epsilon(z,\cdot)d\mu(y,z)$.

Owing to Assumption~\ref{ass:kernel}, one has
\[
F_{\rm aux}^\epsilon[\mu,1]\ge m_\epsilon \int d\mu=m_\epsilon>0,
\]
for all $\mu\in\mathcal{P}(\T^d)$. In addition, for all $k\in\N_0$, one has
\[
\nabla^k F_{\rm aux}^\epsilon[\mu,\psi]=\iint \psi(y,z)\nabla^kK_\epsilon(z,\cdot) d\mu(y,z),
\]
thus, one obtains
\[
\|F_{\rm aux}^{\epsilon}[\mu,\psi]\|_{\mathcal{C}^k}\le \|\psi\|_{\mathcal{C}^0}M_\epsilon^{(k)}<\infty,
\]
owing to Assumption~\ref{ass:kernel}.

Using the estimate above with $\psi=\nabla_zV$ and $\psi=1$, it is then straightforward to deduce that
\[
\|F^\epsilon[\mu]\|_{\mathcal{C}^k}=\|\frac{F_{\rm aux}[\mu,\nabla_zV]}{F_{\rm aux}[\mu,1]}\|_{\mathcal{C}^k}\le C_{\epsilon,k}.
\]
This concludes the proof of the estimates for $F^\epsilon[\mu]$. To prove the estimates for $A^\epsilon[\mu]$, observe that $\tilde{A}^\epsilon[\mu]$ solves the Euler-Lagrange equation associated with the minimization problem in~\eqref{eq:def-Aepsilon},
\[
\Delta \tilde{A}^\epsilon[\mu]={\rm div}\bigl(F^\epsilon[\mu]\bigr).
\]
Using the result proved above, and standard elliptic regularity theory and Sobolev embeddings, one obtains the required estimates for $\tilde{A}^\epsilon[\mu]$: for all $\epsilon\in(0,1]$ and $k\in\N_0$, there exists $C_{\epsilon,k}\in(0,\infty)$ such that for all $\mu\in\mathcal{P}(\T^d)$,
\[
\|\tilde{A}^\epsilon[\mu]\|_{\mathcal{C}^k(\T^m,\T)}\le C_{\epsilon,k}.
\]
Since $A^\epsilon[\mu]$ and $\tilde{A}^\epsilon[\mu]$ only differ by an additive constant, it only remains to prove that
\[
\|A^\epsilon[\mu]\|_{\mathcal{C}^0(\T^m,\T)}\le C_{\epsilon,0}.
\]
This is a straightforward consequence of the estimate $\|\tilde{A}^\epsilon[\mu]\|_{\mathcal{C}^0(\T^m,\T)}\le C_{\epsilon,0}$ and of~\eqref{eq:def-Aepsilon}.

This concludes the proof of Lemma~\ref{lem:bounds-F_A}.
\end{proof}

\begin{lemma}\label{lem:Lip-F_A}
For all $\epsilon\in (0,1]$ and $k\in\N_0$, there exists $L_{\epsilon,k}\in(0,\infty)$ such that, for all $\mu_1,\mu_2\in\mathcal{P}(\T^d)$, one has
\[
\|F^\epsilon[\mu_2]-F^\epsilon[\mu_1]\|_{\mathcal{C}^k(\T^m,\R^m)}+\|A^\epsilon[\mu_2]-A^\epsilon[\mu_1]\|_{\mathcal{C}^k(\T^m,\R)}\le L_{\epsilon,k}\bigl(d_{\rm TV}(\mu_1,\mu_2)\wedge d_{\mathcal{W}_1}(\mu_1,\mu_2)\bigr).
\]
\end{lemma}

\begin{proof}[Proof of Lemma~\ref{lem:Lip-F_A}]
First, observe that
\begin{align*}
F^\epsilon[\mu_2]-F^\epsilon[\mu_1]&=\frac{\iint \nabla_zV(y,z)K_\epsilon(z,\cdot)d(\mu_2-\mu_1)(y,z)}{\iint K_\epsilon(z,\cdot)d\mu_2(y,z)}\\
&-\frac{\iint \nabla_z V(y,z)K_\epsilon(z,\cdot)d\mu_1(y,z) \iint K_\epsilon(z,\cdot)d(\mu_2-\mu_1)(y,z)}{\iint K_\epsilon(z,\cdot)d\mu_1(y,z)\iint K_\epsilon(z,\cdot)d\mu_2(y,z)}.
\end{align*}

Using the characterizations of total variation and Wasserstein distances and the regularity properties of $V$ and $K_\epsilon$ (Assumption~\ref{ass:kernel}), proceeding as in the proof of Lemma~\ref{lem:bounds-F_A} then yields
\[
\|F^\epsilon[\mu_2]-F^\epsilon[\mu_1]\|_{\mathcal{C}^k(\T^m,\T^m)}\le L_{\epsilon,k}d_(\mu_1,\mu_2),
\]
for all $\mu_1,\mu_2\in\mathcal{P}(\T^d)$, with $L_{\epsilon,k}\in(0,\infty)$, with $d=d_{\mathcal{W}_1}$ and $d=d_{TV}$.

It remains to apply the same arguments as in the proof of Lemma~\ref{lem:bounds-F_A} to obtain
\[
\|\tilde{A}^\epsilon[\mu_2]-A^\epsilon[\mu_1]\|_{\mathcal{C}^k(\T^m,\T)}+\|A^\epsilon[\mu_2]-A^\epsilon[\mu_1]\|_{\mathcal{C}^k(\T^m,\T)}\le L_{\epsilon,k}d(\mu_1,\mu_2),
\]
which concludes the proof of Lemma~\ref{lem:Lip-F_A}.
\end{proof}

\subsection{Well-posedness}

Let $T\in(0,\infty)$ be an arbitrary positive real number. Introduce the Banach spaces
\[
\mathcal{C}([0,T],\T^d)~,\quad E=L^2\bigl(\Omega,\mathcal{C}([0,T],\T^d)\bigr),
\]
equipped with the norms defined by
\[
\|x\|_\alpha=\underset{0\le t\le T}\sup~e^{-\alpha t}|x(t)|~,\quad   \vvvert X\vvvert_{\alpha}=\Bigl(\E\bigl[\|X\|_\alpha^2\bigr]\Bigr)^{\frac12},
\]
depending on the auxiliary parameter $\alpha\in(0,\infty)$. Let $\Phi:E\to E$ be defined as follows: for all $x=\bigl(y_t,z_t\bigr)_{t\ge 0}$, let $\mu_t^x=\frac{1}{1+t}\bigl(\mu_0+\int_{0}^{t}\delta_{x_s}ds\bigr)$ and $A_t^x=A^\epsilon[\mu_t^x]$, for all $t\ge 0$. Then $X=\Phi(x)$ is the solution $X=\bigl(Y(t),Z(t))_{t\ge 0}$ of
\[
\begin{cases}
dY(t)=-\nabla_yV(y_t,z_t)dt+\sqrt{2}dW^{(d-m)}(t),\\
dZ(t)=-\nabla_zV(y_t,z_t)dt+\nabla A_t^x(z_t)dt+\sqrt{2}dW^{(d)}(t),
\end{cases}
\]
with initial condition $(Y(0),Z(0))=x_0\in \T^d$, which is fixed.

If $\alpha$ is sufficiently large, then the mapping $\Phi$ is a contraction, due to Lemma~\ref{lem:contraction-wellposed} stated below.
\begin{lemma}\label{lem:contraction-wellposed}
There exists $C\in(0,\infty)$ such that for all $\alpha\in(0,\infty)$, and for all $x^1,x^2\in E$, 
\[
\vvvert\Phi(x_2)-\Phi(x_1)\vvvert_{\alpha}\le \frac{C}{\alpha}\vvvert x_2-x_1\vvvert_{\alpha}.
\]
\end{lemma}

\begin{proof}[Proof of Lemma~\ref{lem:contraction-wellposed}]
Let $x^1=(y^1,z^1)$ and $x^2=(y^2,z^2)$ be two elements of $E$, and set $X^1=\Phi(x^1)$, $X^2=\Phi(x^2)$. Then
\[
\frac{d}{dt}\bigl(Y^2(t)-Y^1(t)\bigr)=\nabla_y V(y_t^1,z_t^1)-\nabla_y V(y_t^2,z_t^2)
\]
and
\[
\frac{d}{dt}\bigl(Z^2(t)-Z^1(t)\bigr)=\nabla_z V(y_t^1,z_t^1)-\nabla_z V(y_t^2,z_t^2)+\nabla A_t^2(z_t^2)-\nabla A_t^1(z_t^1),
\]
where $A_t^i=A^\epsilon[\mu_t^i]$ and $\mu_t^i=\frac{1}{1+t}(\mu_0+\int_{0}^{t}\delta_{x_s^i}ds)$.

First, since $V$ is of class $\mathcal{C}^2$, for all $t\ge 0$, one has the almost sure estimate
\begin{align*}
e^{-\alpha t}|Y^2(t)-Y^1(t)|&\le Ce^{-\alpha t}\int_{0}^{t}\bigl(|y_s^2-y_s^1|+|z_s^2-z_s^1|\bigr)ds\\
&\le Ce^{-\alpha t}\int_{0}^{t}e^{\alpha s}ds \|x^2-x^1\|_{\alpha}\\
&\le \frac{C}{\alpha}\|x^2-x^1\|_{\alpha}.
\end{align*}
Second, similarly one has, for all $t\ge 0$,
\begin{align*}
e^{-\alpha t}|Z^2(t)-Z^1(t)|&\le \frac{C}{\alpha}\|x^2-x^1\|_{\alpha}+e^{-\alpha t}\int_{0}^{t}|\nabla A_s^2(z_s^2)-\nabla A_s^1(z_s^1)|ds\\
&\le \frac{C}{\alpha}\|x^2-x^1\|_{\alpha}+e^{-\alpha t}\int_{0}^{t}|\nabla A_s^2(z_s^2)-\nabla A_s^2(z_s^1)|ds+e^{-\alpha t}\int_{0}^{t}|\nabla A_s^2(z_s^1)-\nabla A_s^1(z_s^1)|ds\\
&\le \frac{C}{\alpha}\|x^2-x^1\|_{\alpha}+e^{-\alpha t}\int_{0}^{t}\|A_s^2-A_s^1\|_{\mathcal{C}^1}ds,
\end{align*}
owing to Lemma~\ref{lem:bounds-F_A}. In addition, owing to Lemma~\ref{lem:Lip-F_A}, one has
\begin{align*}
\|A_s^2-A_s^1\|_{\mathcal{C}^1}&=\|A^\epsilon[\mu_s^2]-A^\epsilon[\mu_s^1]\|_{\mathcal{C}^1}\\
&\le L_{\epsilon,1}d_{\mathcal{W}_1}(\mu_s^1,\mu_s^2)\le L_{\epsilon,1}\int_{0}^{s}|x^2(r)-x^1(r)|dr\\
&\le L_{\epsilon,1}\int_{0}^{s}e^{\alpha r}dr \|x^2-x^1\|_{\alpha}\\
&\le \frac{L_{\epsilon,1}}{\alpha}e^{\alpha s}\|x^2-x^1\|_{\alpha}.
\end{align*}
Finally, one obtains the almost sure estimate,
\[
\|\Phi(x^2)-\Phi(x^1)\|_{\alpha}=\underset{t\ge 0}\sup~e^{-\alpha t}|X^2(t)-X^1(t)|\le \frac{C}{\alpha}\|x^2-x^1\|_{\alpha},
\]
then taking expectation concludes the proof of Lemma~\ref{lem:contraction-wellposed}.
\end{proof}

The proof of Proposition~\ref{propo:wellposed} is then straightforward.
\begin{proof}[Proof of Proposition~\ref{propo:wellposed}]
Observe that the following claims are satisfied.
\begin{itemize}
\item Owing to Lemma~\ref{lem:bounds-F_A}, for all $x\in E$, one has the almost sure estimate $\underset{t\ge 0}\sup~\|\nabla A_t^x\|_{\mathcal{C}^0}\le C_{\epsilon,0}$, and owing to Lemma~\ref{lem:Lip-F_A}, the mapping $t\mapsto A_t^x$ is Lipschitz continuous. Thus the mapping $\Phi$ is well-defined.
\item The process $\bigl(Y(t),Z(t),A_t,\mu_t\bigr)_{t\ge 0}$ solves~\eqref{eq:ABF} if and only if $X=(Y,Z)$ is a fixed point of $\Phi$.
\item The mapping $\Phi:E\to E$ is a contraction if $\alpha$ is sufficiently large, and admits a unique fixed point $X$, owing to Lemma~\ref{lem:contraction-wellposed}.
\end{itemize}
Since the initial conditions $x_0$ and $\mu_0$, and the time $T\in(0,\infty)$ are arbitrary, these arguments imply that the global well-posedness of~\eqref{eq:ABF} and this concludes the proof.
\end{proof}

\section{The limiting flow}\label{sec:ODE}

Define the mapping $\Pi^\epsilon:\mu\in\mathcal{P}(\T^d)\mapsto \Pi^\epsilon[\mu]\in \mathcal{P}(\T^d)$, for $\epsilon\in(0,1]$, as follows:
\[
\Pi^\epsilon[\mu]=Z^{\epsilon}[\mu]^{-1}e^{-V(y,z)+A^\epsilon[\mu](z)}dydz,
\]
with $Z^\epsilon[\mu]=\iint e^{-V(y,z)+A^\epsilon[\mu](z)}dydz$. The notation $V_\mu^\epsilon(y,z)=V(y,z)-A^\epsilon[\mu](z)$ is used in the sequel. The probability measure $\Pi^{\epsilon}[\mu]$ is the unique invariant distribution for the system
\[
\begin{cases}
dY_t^A=-\nabla_yV(Y_t^A,Z_t^A)dt+\sqrt{2}dW_t^{(d-m)},\\
dZ_t^A=-\nabla_zV(Y_t^A,Z_t^A)dt+\nabla A(Z_t^A)dt+\sqrt{2}dW_t^{(m)}
\end{cases}
\]
with $A=A^\epsilon[\mu]$. With notations used above, $\Pi^\epsilon[\mu]=\mu_\star^{A^{\epsilon}[\mu]}$.

The objectives of this section are twofold. First, one proves that, for every $\pi\in\mathcal{P}(\T^d)$, there exists a unique solution $\bigl(\Phi^\epsilon(t,\pi)\bigr)_{t\ge 0}$ of the equation
\[
\Phi^{\epsilon}(t,\pi)=e^{-t}\pi+\int_{0}^{t}e^{s-t}\Pi^\epsilon[\Phi^\epsilon(s,\pi)]ds.
\]
In addition, $\pi_t^\epsilon=\Phi^\epsilon(t,\pi)$ solves, in a weak sense, the following ordinary differential equation
\[
\dot{\pi}_t^\epsilon=\Pi^\epsilon[\pi_t^\epsilon]-\pi_t^\epsilon~,\quad \pi_0^\epsilon=\pi.
\]
Second, one relates the properties of the empirical measure $\bigl(\mu_t\bigr)_{t\ge 0}$ in the regime $t\to\infty$, with the behavior of the limit flow, using the notion of Asymptotic Pseudo-Trajectories.

\subsection{Well-posedness of the limiting flow}

Let $M^\epsilon=\underset{\mu\in\mathcal{P}(\T^d)}\sup~\|A^\epsilon[\mu]\|_{\mathcal{C}^0(\T^m,\R)}$, and $M_\star=\|A_\star\|_{\mathcal{C}^0(\T^m,\R)}$. Note that $M^\epsilon<\infty$ due to Lemma~\ref{lem:bounds-F_A}. Recall that $L_{0,\epsilon}$ is defined in Lemma~\ref{lem:Lip-F_A}.

\begin{lemma}\label{lem:Lip_PiEpsilon}
Let $L(\epsilon)=2L_{\epsilon,0}e^{4(M^\epsilon+M_\star)}$. Then for all $\mu^1,\mu^2\in\mathcal{P}(\T^d)$, one has
\[
d_{TV}\bigl(\Pi^\epsilon[\mu^1],\Pi^\epsilon[\mu^2]\bigr)\le L(\epsilon)d_{TV}(\mu^1,\mu^2).
\]
\end{lemma}

\begin{proof}[Proof of Lemma~\ref{lem:Lip_PiEpsilon}]

\begin{align*}
d_{TV}\bigl(\Pi^\epsilon[\mu^1],\Pi^\epsilon[\mu^2]\bigr)&=\iint_{\T^d} e^{-V(y,z)}\big|\frac{e^{A^\epsilon[\mu^1](z)}}{Z^\epsilon[\mu^1]}-\frac{e^{A^\epsilon[\mu^2](z)}}{Z^\epsilon[\mu^2]}\big|dydz\\
&=\int_{\T^m}e^{-A_\star(z)}\big|\frac{e^{A^\epsilon[\mu^1](z)}}{Z^\epsilon[\mu^1]}-\frac{e^{A^\epsilon[\mu^2](z)}}{Z^\epsilon[\mu^2]}\big|dz\\
&\le \int_{\T^m}\frac{e^{-A_\star(z)}}{Z^\epsilon[\mu^1]}\big|e^{A^\epsilon[\mu^1](z)}-e^{A^\epsilon[\mu^2](z)}\big|dz\\
&~+\int_{\T^m}\frac{e^{A^\epsilon[\mu^2](z)-A_\star(z)}}{Z^\epsilon[\mu^1]Z^\epsilon[\mu^2]}dz\big|Z^\epsilon[\mu^1]-Z^\epsilon[\mu^2]\big|.
\end{align*}
Using the lower bound
\[
Z^\epsilon[\mu]=\iint_{\T^d} e^{-V(y,z)+A^\epsilon[\mu](z)}dydz=\int_{\T^m} e^{-A_\star(z)+A^\epsilon[\mu](z)}dz\ge e^{-M_\star-M^\epsilon},
\]
and the upper bound
\[
\big|Z^\epsilon[\mu^1]-Z^\epsilon[\mu^2]\big|\le e^{M^\epsilon+M_\star}\int_{\T^m}|A^\epsilon[\mu^1](z)-A^\epsilon[\mu^2](z)|dz,
\]
one obtains
\begin{align*}
d_{TV}\bigl(\Pi^\epsilon[\mu^1],\Pi^\epsilon[\mu^2]\bigr)&\le 2e^{4(M^\epsilon+M_\star)}\int_{\T^m}|A^\epsilon[\mu^1](z)-A^\epsilon[\mu^2](z)|dz\\
&\le 2e^{4(M^\epsilon+M_\star)}\|A^\epsilon[\mu^1]-A^\epsilon[\mu_2]\|_{\mathcal{C}^0}\\
&\le 2L_{\epsilon,0}e^{4(M^\epsilon+M_\star)}d_{\rm TV}(\mu_1,\mu_2),
\end{align*}
where the last inequality follows from Lemma~\ref{lem:Lip-F_A}. This concludes the proof of Lemma~\ref{lem:Lip_PiEpsilon}.
\end{proof}

\begin{propo}
Let $\pi\in\mathcal{P}(\T^d)$. Then there exists a unique solution $\bigl(\Phi^\epsilon(t,\pi)\bigr)_{t\ge 0}$, with values in $\mathcal{C}\bigl([0,\infty),\mathcal{P}(\T^d)\bigr)$ (where $\mathcal{P}(\T^d)$ is equipped with the total variation distance $d_{TV}$), of the equation
\[
\Phi^{\epsilon}(t,\pi)=e^{-t}\pi+\int_{0}^{t}e^{s-t}\Pi^\epsilon[\Phi^\epsilon(s,\pi)]ds.
\]
\end{propo}

\begin{proof}
Uniqueness is a straightforward consequence of Lemma~\ref{lem:Lip_PiEpsilon} and of Gronwall Lemma.

Existence is obtained using a Picard iteration argument. Precisely, introduce the mapping $\Psi:\mathcal{C}\bigl([0,\infty),\mathcal{P}(\T^d)\bigr)\to \mathcal{C}\bigl([0,\infty),\mathcal{P}(\T^d)\bigr)$, be defined by
\[
\Psi(\pi)(t)=e^{-t}\pi+\int_{0}^{t}e^{s-t}\Pi^\epsilon[\pi_s]ds,
\]
for $\pi=\bigl(\pi_t\bigr)_{t\ge 0}$.

Let $d_\alpha(\pi^1,\pi^2)=\underset{t\ge 0}\sup~e^{-\alpha t}d_{TV}(\pi_t^1,\pi_t^2)$, where $\alpha>0$ is chosen below. Then, using Lemma~\ref{lem:Lip_PiEpsilon}, one has
\[
d_\alpha\bigl(\Psi(\pi^1),\Psi(\pi^2)\bigr)\le \frac{L(\epsilon)}{\alpha}d_\alpha(\pi^1,\pi^2).
\]
Choose $\alpha=2L(\epsilon)$, and define
\[
\pi^0=\bigl(\pi_t^0=\pi\bigr)_{t\ge 0}~,\quad \pi^{n+1}=\Psi(\pi^n),~n\ge 0,
\]
using the Picard iteration method. Let $T\in(0,\infty)$ be an arbitrary positive real number. Since $\mathcal{C}\bigl([0,T],\mathcal{P}(\T^d)\bigr)$ is a complete metric space (equipped with the distance $d_\alpha$), then $\bigl(\pi^n\bigr)_{n\in\N}$ converges when $n\to\infty$, and the limit $\pi^\infty$ solves the fixed point equation $\pi^\infty=\Psi(\pi^\infty)$, which proves the existence of a solution, and concludes the proof.
\end{proof}

By construction, the flow $\Phi^\epsilon:\R^+\times\mathcal{P}(\T^d)\to\mathcal{P}(\T^d)$ is continuous, when $\mathcal{P}(\T^d)$ is equipped with the total variation distance $d_{\rm TV}$. Adapting the proof of~\cite[Lemma~3.3]{BLR02}, one checks that it is also a continuous mapping when $\mathcal{P}(\T^d)$ is equipped with the distance $d_w$.

\subsection{The asymptotic pseudotrajectory property}

Recall that a continuous function $\zeta:\R^+\to\mathcal{P}(\T^d)$ is an asymptotic pseudotrajectory for $\Phi^\epsilon$, if one has
\[
\underset{s\in[0,T]}\sup~d_w\bigl(\zeta(t+s),\Phi^\epsilon(s,\zeta(t))\bigr)\underset{t\to\infty}\to 0,
\]
for all $T\in\R^+$. See for instance~\cite{B99} for details.

The following result is the rigorous formulation of the link between the dynamics of the empirical measures $\mu_t$ in the ABF algorithm, and of the limit flow.
\begin{theo}\label{theo:APT}
The process $\bigl(\mu_{e^t}\bigr)_{t\ge 0}$ is almost surely an asymptotic pseudotrajectory for $\Phi^\epsilon$.
\end{theo}

The proof requires auxiliary notations and results.  For every $\epsilon>0$ and $\mu\in\mathcal{P}(\T^d)$, let
\[
V_\mu^\epsilon(y,z)=V(y,z)-A^\epsilon[\mu](z),
\]
and define the infinitesimal generator
\[
\mathcal{L}_\mu^\epsilon=\Delta -\nabla V_\mu^\epsilon\cdot\nabla.
\]
Introduce the projection operator defined by $K_\mu^\epsilon f=f-\int fd\Pi^\epsilon[\mu]$
and let $\bigl(P_t^{\epsilon,\mu}\bigr)_{t\ge 0}$ be the semi-group generated by $\mathcal{L}_\mu^\epsilon$ on $L^2(\T^d)$.  Finally, let
\[
Q_\mu^\epsilon=\int_0^\infty P_t^{\epsilon,\mu}K_\mu^\epsilon dt\,.
\]
Then one has the following result.
\begin{lemma}
For every $\epsilon>0$, there exists $C_\epsilon\in(0,\infty)$, such that
\begin{eqnarray}\label{eq:borneQmu}
\|Q_\mu^\epsilon f\|_{\mathcal{C}^1}\le C_\epsilon\|f\|_{\mathcal{C}^0}\,,
\end{eqnarray}
for all $f\in\mathcal{C}^0(\T^d,\R)$ and all $\mu\in\mathcal{P}(\T^d)$. Moreover, $\mathcal{L}_\mu^\epsilon K_\mu^\epsilon=-K_\mu^\epsilon$.
\end{lemma}

\begin{proof}
Remark that, from Lemma \ref{lem:bounds-F_A}, $V_\mu^\varepsilon\in\mathcal C^\infty(\T^d)$, from which it is classical to see that $P_t^{\epsilon,\mu} f\in\mathcal C^\infty(\T^d)$ for all $f\in \mathcal{C}^\infty(\T^d)$. In particular, $\mathcal{C}^\infty(\T^d)$ is a core for $\mathcal L^{\epsilon,\mu}$, see \cite[Section 3.2]{BakryGentilLedoux} and thus it is enough to prove the result for $f\in\mathcal C^\infty(\T^d)$.

As a first step, for all $\epsilon \in (0,1]$ there exists $R_\epsilon>0$ such that for all $\mu\in \mathcal{P}(\T^d)$, $\Pi^\epsilon[\mu]$ satisfies a log-Sobolev inequality and a Sobolev inequality both with constant $R_\epsilon$, in the sense that for all positive $f\in\mathcal C^\infty(\T^d)$,
\begin{eqnarray*}
\int_{\mathbb T^d} f \ln f  d \Pi^\epsilon[\mu]- \int_{\mathbb T^d}  f d \Pi^\epsilon[\mu] \ln \int_{\mathbb T^d}  f  d \Pi^\epsilon[\mu]  & \leqslant & R_\epsilon \int_{\mathbb T^d}  \frac{|\nabla f|^2}{f} d \Pi^\epsilon[\mu]\\
\| f\|_{L^p(\Pi^\epsilon[\mu])}^2 & \leqslant & R_\epsilon  \| f\|_{H^1(\Pi^\epsilon[\mu])}^2\,,
\end{eqnarray*}
where $p= \frac{2d}{d-2}$. Indeed, from Lemma \ref{lem:bounds-F_A}, the density of $\Pi^\epsilon[\mu]$ with respect to the Lebesgue measure is bounded above and below away from zero uniformly in $\mu\in\mathcal P(\T^d)$. The inequalities are then obtained by  a perturbative argument from those satisfied by the Lebesgue measure, see \cite[Proposition 5.1.6]{BakryGentilLedoux}).

As a second step, these inequalities imply the following estimates: for all $\epsilon \in (0,1]$ there exists $R_\epsilon'>0$ such that for all $\mathcal{P}(\T^d)$, $f\in\mathcal C^{\infty}(\T^d)$ and $t\geq 0$, 
\begin{eqnarray*}
\| P_t^{\epsilon,\mu} K_\mu^\epsilon f \|_{L^2(\Pi^\varepsilon[\mu])} & \leqslant & e^{-R_\epsilon t/2} \|  K_\mu^\epsilon f \|_{L^2(\Pi^\varepsilon[\mu])}\\
\| P_t^{\epsilon,\mu} f \|_{\infty} & \leqslant & \frac{R_\epsilon'}{\max(1, t^{d/2})} \|   f \|_{L^2(\Pi^\varepsilon[\mu])}\\
\| \nabla P_t^{\epsilon,\mu} f \|_{\infty} & \leqslant & \frac{R_\epsilon'}{\max(1,\sqrt t)} \|   f \|_{\infty}.
\end{eqnarray*}
Indeed, the first estimate is a usual consequence of the Poincar\'e inequality, implied by the log-Sobolev one (see \cite[Theorem 4.2.5 and Proposition 5.1.3]{BakryGentilLedoux}). The second one, namely the ultracontractivity of the semi-group, is a consequence of the Sobolev inequality (see \cite[Theorem 6.3.1]{BakryGentilLedoux}). The last one can be established thanks to the Bakry-Emery calculus (see \cite[Section 1.16]{BakryGentilLedoux} for an introduction), by showing that $\mathcal{L}_\mu^\epsilon$ satisfies a curvature estimate. More precisely, denote
\begin{eqnarray*}
\Gamma^{\epsilon,\mu}(f,g) & = & \frac12 \left ( \mathcal{L}_\mu^\epsilon(fg) - f \mathcal{L}_\mu^\epsilon g - g \mathcal{L}_\mu^\epsilon  f\right)\\ 
\Gamma_2^{\epsilon,\mu}(f) & = & \frac12 \Gamma^{\epsilon,\mu}(f) - \Gamma^{\epsilon,\mu}(f, \mathcal{L}_\mu^\epsilon f ),
\end{eqnarray*}
with $\Gamma^{\epsilon,\mu}(f) := \Gamma^{\epsilon,\mu}(f,f)$. Straightforward computations yield
\begin{eqnarray*}
\Gamma^{\epsilon,\mu}(f) & = &  |\nabla f|^2 \\ 
\Gamma_2^{\epsilon,\mu} (f) & \geqslant & -   |\nabla^2 V_\mu^\epsilon|   |\nabla f|^2 \ \geqslant \ -c_\epsilon \Gamma^{\epsilon,\mu}(f)
\end{eqnarray*}
for some $c_\epsilon>0$ which does not depend on $\mu\in\mathcal P(\T^d)$ thanks to Lemma \ref{lem:bounds-F_A}. According to  \cite[Theorem 4.7.2]{BakryGentilLedoux}, this implies that 
\[\Gamma^{\epsilon,\mu}(P_t^{\epsilon,\mu} f) \ \leqslant \  \left( \frac{1-e^{-c'_\epsilon t}}{c'_\epsilon}\right)^{-1} P_f^{\epsilon,\mu} f^2 \ \leqslant \  \left( \frac{1-e^{-c'_\epsilon t}}{c'_\epsilon}\right)^{-1}\| f\|_\infty^2, \]
which concludes the proof of the third estimate.

As a third step, we bound (using that $\|P_t^{\epsilon,\mu}f\|_\infty \leqslant \|f\|_\infty$ for all $t\geqslant 0$)
\begin{eqnarray*}
  \int_0^\infty \| P_t^{\epsilon,\mu}K_\mu^\epsilon f\|_\infty dt & \leqslant &   \int_0^1 \|K_\mu^\epsilon f\|_\infty dt +   \int_1^\infty \| P_t^{\epsilon,\mu}K_\mu^\epsilon f\|_\infty dt \\
  & \leqslant & 2\|f\|_\infty +    R_\epsilon' \int_1^\infty \| P_{t-1}^{\epsilon,\mu}K_\mu^\epsilon f\|_{L^2(\Pi^\epsilon[\mu])} dt \\
    & \leqslant & 2\|f\|_\infty +    R_\epsilon' \int_0^\infty  e^{-R_\epsilon s/2} \|  K_\mu^\epsilon f\|_{L^2(\Pi^\epsilon[\mu])} dt \\
    & \leqslant & \left(2 + \frac{4R_\epsilon'}{R_\epsilon}\right)\|f\|_\infty\,,
\end{eqnarray*} 
and similarly
\begin{eqnarray*}
  \int_0^\infty \| \nabla P_t^{\epsilon,\mu}K_\mu^\epsilon f\|_\infty dt & \leqslant &   \int_0^2 \frac{R_\epsilon'}{\max(1,\sqrt t)}   \|   K_\mu^\epsilon f\|_\infty dt + R_\epsilon'  \int_2^\infty    \|   P_{t-1}^{\epsilon,\mu}K_\mu^\epsilon f\|_\infty dt   \\  
  & \leqslant & 6 R_\epsilon' \|f\|_\infty  +    R_\epsilon'^2 \int_0^\infty  e^{-R_\epsilon s/2} \|  K_\mu^\epsilon f\|_{L^2(\Pi^\epsilon[\mu])} dt \\
    & \leqslant & \left(6 R_\epsilon' + \frac{4R_\epsilon'^2}{R_\epsilon}\right)\|f\|_\infty\,,
\end{eqnarray*} 
from which $Q_\mu^\epsilon f$ is well defined for $f\in\mathcal C^\infty(\T^d)$ and satisfies \eqref{eq:borneQmu} for some $C_\epsilon$. Finally,
\begin{eqnarray*}
\mathcal{L}_\mu^\epsilon Q_\mu^\epsilon f &  = & \int_0^\infty \mathcal{L}_\mu^\epsilon P_t^{\epsilon,\mu} K_\mu^\epsilon f dt \\
& = & \int_0^\infty \partial_t\left( P_t^{\epsilon,\mu} K_\mu^\epsilon f\right) dt  \ = \ - K_\mu^\epsilon f\,.
\end{eqnarray*}
\end{proof}

\begin{proof}[Proof of Theorem \ref{theo:APT}]
First, note that the claim is equivalent to the following statement (see~\cite[Proposition~3.5]{BLR02}):
\[
\underset{s\in[0,T]}\sup~|\varepsilon_t(s)f|\underset{t\to\infty}\to 0,
\]
for all $f\in\mathcal{S}$ and $T\in\mathbb{Q}^+$, where
\[
\varepsilon_t(s)=\int_{e^{t}}^{e^{t+s}}\frac{\delta_{X_\tau}-\Pi^\epsilon[\mu_\tau]}{\tau}d\tau.
\]

Using a Borel-Cantelli argument, and the fact that $\mathcal{S}$ is a countable set, it is sufficient to establish that there exists $C_\epsilon\in(0,\infty)$, such that
\[
\E\bigl[\underset{s\in[0,T]}\sup~|\varepsilon_t(s)f|^2\bigr]\le C_\epsilon e^{-t}\|f\|_{\mathcal{C}^0}^2,
\]
for all $t\ge 0$ and $f\in\mathcal{S}$.

Let $f\in\mathcal{S}$ and introduce the function $F:(0,\infty)\times \T^d\to\R$ defined by $F(t,x)=t^{-1}Q_{\mu_t}^\epsilon f$. Then $F$ is of class $\mathcal{C}^{1,2}$ on $(0,\infty)\times \T^d$. Indeed, first, it is straightforward to check that $t\mapsto F^\epsilon[\mu_t]\in \mathcal{C}^k(\T^d,\R^m)$ is of class $\mathcal{C}^1$, for all $k\in\N_0$, since $t\mapsto \mu_t\in\mathcal{P}(\T^d)$ (equipped with the Wasserstein distance) is of class $\mathcal{C}^1$. Second, $A^\epsilon[\mu]$ is solution of the Euler-Lagrange equation $\Delta A^\epsilon[\mu]={\rm div}(F^\epsilon[\mu])$, which establishes that $t\mapsto A^\epsilon[\mu_t]\in\mathcal{C}^k(\T^m,\R)$ is also of class $\mathcal{C}^1$. Finally, it remains to apply standard arguments to establish the $\mathcal{C}^1$ regularity of $t\mapsto Q_{\mu_t}^\epsilon f$.

Applying It\^o formula yields, for all $t\ge 0$ and $s\in[0,T]$, the equality
\begin{align*}
F(e^{t+s},X_{e^{t+s}})&=F(e^t,X_{e^t})+\int_{e^t}^{e^{t+s}}\bigl(\partial_\tau+\mathcal{L}_{\mu_\tau}^\epsilon\bigr)F(\tau,X_\tau)d\tau+\sqrt{2}\int_{e^t}^{e^{t+s}}\langle \nabla F(\tau,X_\tau),dW(\tau)\rangle. 
\end{align*}
Observing that $\mathcal{L}_{\mu_\tau}^\epsilon F(\tau,X_\tau)=\tau^{-1}\mathcal{L}_{\mu_\tau}^\epsilon Q_{\mu_\tau}^\epsilon(X_\tau) f=-\tau^{-1}\bigl(f(X_\tau)-\int fd\Pi^\epsilon[\mu_\tau]\bigr)$, one obtains
\[
\varepsilon_t(s)f=\varepsilon_t^1(s)f+\varepsilon_t^2(s)f+\varepsilon_t^3(s)f+\varepsilon_t^4(s)f,
\]
where
\begin{align*}
\varepsilon_t^1(s)f&=e^{-t}\Bigl(Q_{\mu_t}^\epsilon f-e^{-s}Q_{\mu_{t+s}}^\epsilon f\Bigr),\\
\varepsilon_t^2(s)f&=\int_{e^{t}}^{e^{t+s}}-\tau^{-2}Q_{\mu_t\tau}^\epsilon f(X_\tau)d\tau,\\
\varepsilon_t^3(s)f&=\int_{e^t}^{e^{t+s}}\tau^{-1}\frac{d}{d\tau}Q_{\mu_\tau}^\epsilon f(X_\tau)d\tau,\\
\varepsilon_t^4(s)f&=\sqrt{2}\int_{e^{t}}^{e^{t+s}} \tau^{-1}\langle \nabla Q_{\mu_\tau}^{\epsilon}f(X_\tau),dW(\tau)\rangle.
\end{align*}

First, it is straightforward to check that the error terms $\varepsilon_t^1(s)f$ and $\varepsilon_t^2(s)f$ are upper estimated as follows: almost surely,
\[
\underset{0\le s\le T}\sup~|\varepsilon_t^1(s)f|+\underset{0\le s\le T}\sup~|\varepsilon_t^2(s)f|\le C_\epsilon e^{-t}\|f\|_\infty.
\]
To treat the error term $\varepsilon_t^3(s)f$, it suffices to upper estimate the Lipschitz constant of $t\mapsto Q_{\mu_t}^\epsilon f$. Let $t_1,t_2\in(0,\infty)$, then one has
\begin{align*}
K_{\mu_{t_1}}^\epsilon f-K_{\mu_{t_2}}^\epsilon f&=\mathcal{L}_{\mu_{t_2}}^\epsilon Q_{\mu_{t_2}}^\epsilon f-\mathcal{L}_{\mu_{t_1}}^\epsilon Q_{\mu_{t_1}}^\epsilon f\\
&=\mathcal{L}_{\mu_{t_1}}^\epsilon\Bigl(Q_{\mu_{t_2}}^\epsilon f-Q_{\mu_{t_1}}^{\epsilon}f\Bigr)+\Bigl(\mathcal{L}_{\mu_{t_2}}^\epsilon-\mathcal{L}_{\mu_{t_1}}^\epsilon\Bigr)Q_{\mu_{t_2}}^\epsilon f,
\end{align*}
thus one obtains
\[
Q_{\mu_{t_2}}^\epsilon f-Q_{\mu_{t_1}}^{\epsilon}f=Q_{\mu_{t_1}}^\epsilon \delta_{t_1,t_2}^\epsilon f,
\]
where the auxiliary function $\delta_{t_1,t_2}^\epsilon f$ is defined as
\[
\delta_{t_1,t_2}^\epsilon f=K_{\mu_{t_1}}^\epsilon f-K_{\mu_{t_2}}^\epsilon f-\Bigl(\mathcal{L}_{\mu_{t_2}}^\epsilon-\mathcal{L}_{\mu_{t_1}}^\epsilon\Bigr)Q_{\mu_{t_2}}^\epsilon f,
\]
and satisfies the centering condition $\int \delta_{t_1,t_2}^\epsilon f d\Pi^{\epsilon}[\mu_{t_1}]=\int \mathcal{L}_{\mu_{t_1}}^\epsilon\Bigl(Q_{\mu_{t_2}}^\epsilon f-Q_{\mu_{t_1}}^{\epsilon}f\Bigr) d\Pi^{\epsilon}[\mu_{t_1}]=0$.

One has the estimate
\[
\|Q_{\mu_{t_2}}^\epsilon f-Q_{\mu_{t_1}}^{\epsilon}f\|_\infty\le C_\epsilon \|\delta_{t_1,t_2}^\epsilon f\|_\infty.
\]
On the one hand, one has
\begin{align*}
\|K_{\mu_{t_1}}^\epsilon f-K_{\mu_{t_2}}^\epsilon f\|_\infty&=\big|\int fd\Pi^\epsilon[\mu_{t_1}]-\int fd\Pi^\epsilon[\mu_{t_2}]\big|\\
&\le \|f\|_\infty d_{\rm TV}(\Pi^\epsilon[\mu_{t_1}],\Pi^\epsilon[\mu_{t_2}])\\
&\le L(\epsilon)\|f\|_\infty d_{\rm TV}(\mu_{t_1},\mu_{t_2}),
\end{align*}
owing to Lemma~\ref{lem:Lip_PiEpsilon}.

On the other hand, one has
\begin{align*}
\|\bigl(\mathcal{L}_{\mu_{t_2}}^\epsilon-\mathcal{L}_{\mu_{t_1}}^\epsilon\bigr)Q_{\mu_{t_2}}^\epsilon f\|_\infty&= \|\langle \nabla A^\epsilon[\mu_{t_2}]-\nabla A^\epsilon[\mu_{t_1}],\nabla_z Q_{\mu_{t_2}}^\epsilon f\rangle\|_\infty\\
&\le \|A^\epsilon[\mu_{t_2}]-A^\epsilon[\mu_{t_1}]\|_{\mathcal{C}^1}\|Q_{\mu_{t_2}}^\epsilon f\|_{\mathcal{C}^1}\\
&\le L_{1,\epsilon}C_\epsilon \|f\|_\infty d_{\rm TV}(\mu_{t_1},\mu_{t_2}).
\end{align*}
Finally, it is straightforward to check that
\[
d_{\rm TV}(\mu_{t_1},\mu_{t_2})\le \frac{2|t_2-t_1|}{t_1\wedge t_2},
\]
using the identity $\dot{\mu}_t=\frac{1}{t+r}(\delta_{X_t}-\mu_t)$.

As a consequence, one obtains
\begin{align*}
\underset{0\le s\le T}\sup~|\varepsilon_t^3(s)f|&\le \int_{e^{t}}^{e^{t+T}}\tau^{-1}|\frac{d}{d\tau}Q_{\mu_\tau}^\epsilon f(X_\tau)|d\tau\\
&\le C_\epsilon\int_{e^{t}}^{e^{t+T}}\tau^{-2}d\tau \|f\|_\infty\\
&\le C_\epsilon e^{-t}\|f\|_\infty.
\end{align*}

It remains to deal with the error term $\varepsilon_t^4(s)f$. Using Doob inequality implies
\begin{align*}
\E\bigl[\underset{0\le s\le T}\sup~|\varepsilon_t^4(s)f|^2\bigr]&\le C\int_{e^{t}}^{e^{t+T}}\tau^{-2}\E\bigl[|\nabla Q_{\mu_\tau}^\epsilon f(X_\tau)|^2\bigr] d\tau\\
&\le C_\epsilon e^{-t}\|f\|_\infty^2.
\end{align*}

This concludes the proof of the claim,
\[
\E\bigl[\underset{s\in[0,T]}\sup~|\varepsilon_t(s)f|^2\bigr]\le C_\epsilon e^{-t}\|f\|_{\mathcal{C}^0}^2,
\]
for all $t\ge 0$ and $f\in\mathcal{S}$.

Applying a Borel-Cantelli argument then concludes the proof.
\end{proof}

\section{Proof of Theorem~\ref{theo:Main}}\label{sec:proof-main}

The objective of this section is to give a detailed proof of Theorem~\ref{theo:Main}. There are two main ingredients. The first one is Proposition~\ref{propo:bound} below, which provides a uniform estimate over $\epsilon>0$ for $A^\epsilon[\mu]$, in the $\mathcal{C}^0$ norm (compare with Lemma~\ref{lem:bounds-F_A} where the upper bound may depend on $\epsilon$). The second key ingredient is Proposition~\ref{propo:contractionPitilde}, which states a contraction property for the mapping $\Pi^\epsilon$, for an appropriate distance, for sufficiently small $\epsilon$, when restricted to an attracting set identified below (compare with Lemma~\ref{lem:Lip_PiEpsilon} which is valid on the entire state space, but where no upper bound for $L(\epsilon)$ holds).

Combining these two ingredients provides a candidate for the limit as $t\to\infty$, using a standard Picard iteration argument. Using Theorem~\ref{theo:APT} (asymptotic pseudo-trajectory property) then proves the almost sure convergence of $\mu_t$ to this candidate limit.

\subsection{Uniform estimate}

The following PDE estimate is crucial for the analysis.
\begin{propo}\label{propo:PDE}
Let $m\in\N$. For every $p\in[2,\infty)$, there exists $C_p\in(0,\infty)$, such that the following holds: let $F:\T^m\to\R^m$ be a continuous function, then the solution $A$ of the elliptic PDE $\Delta A={\rm div}(F)$, with the condition $\int A(z)dz=0$, satisfies
\[
\|A\|_{W^{1,p}(\T^m,\R)}\le C_p\|F\|_{\mathcal{C}^0(\T^m,\R^m)},
\]
and if $p>m$, then 
\[
\|A\|_{\mathcal{C}^0(\T^m,\R)}\le C_p\|F\|_{\mathcal{C}^0(\T^m,\R^m)}.
\]
\end{propo}

\begin{proof}
The proof combines three arguments.
\begin{itemize}
\item If $p>m$, then by Sobolev embedding properties, one has $\|A\|_{\mathcal{C}^{0}(\T^m,\R)}\le C_p\|A\|_{W^{1,p}(\T^m,\R)}$, with $C_p\in(0,\infty)$.
\item By the Poincar\'e inequality (using the condition $\int A(z)dz=0$, one has $\|A\|_{W^{1,p}(\T^m,\R)}\le C_p\|\nabla A\|_{L^{p}(\T^m,\R^m)}$, with $C_p\in(0,\infty)$, see~\cite[Theorem~1.13]{Ambrosio}.
\item By elliptic regularity theory, one has $\|\nabla A\|_{L^{p}(\T^m,\R^m)}\le C_p\|F\|_{L^{p}(\T^m,\R^m)}\le C_p\|F\|_{\mathcal{C}^0(\T^m,\R^m)}$, with $C_p\in(0,\infty)$, see~\cite[Theorem~15.12]{Ambrosio}.
\end{itemize}
\end{proof}

\begin{rem}
If $m=1$, the proof is straighforward: indeed for all $z\in\T$, one has the identity $A(z)=\int_0^zF(z')dz'-z\int_0^1F(z')dz'$.
\end{rem}

Using Proposition~\ref{propo:PDE}, one gets the following crucial estimate, which is uniform for $\epsilon>0$ (contrary to those given in Lemmas~\ref{lem:bounds-F_A},\ref{lem:Lip-F_A} and \ref{lem:Lip_PiEpsilon} above).

\begin{propo}\label{propo:bound}
One has the following estimate:
\[
M^0=\underset{\epsilon>0}\sup~\underset{\mu\in\mathcal{P}(\T^d)}\sup~\|A^\epsilon[\mu]\|_{\mathcal{C}^0(\T^m,\R)}<\infty.
\]
\end{propo}

\begin{proof}
Using Proposition~\ref{propo:PDE} above, it suffices to check that
\[
\underset{\epsilon>0}\sup~\underset{\mu\in\mathcal{P}(\T^d)}\sup~\|F^\epsilon[\mu]\|_{\mathcal{C}^0(\T^m,\R^m)}<\infty.
\]
That estimate is a straightforward consequence of the definition~\ref{eq:def-Fepsilon}, of the boundedness of $\nabla_zV$, and of the positivity of the kernel function $K_\epsilon$. 
\end{proof}

\subsection{Attracting set}

Introduce the following notation: for all $B\in\mathcal{C}(\T^m,\R)$, let
\[
d\mu_B(y,z)=\mathcal{Z}_B^{-1}e^{-V(y,z)+B(z)}dydz\in\mathcal{P}(\T^d),
\]
with $\mathcal{Z}_B=\iint e^{-V(y,z)+B(z)}dydz=\int e^{-A_\star(z)+B(z)}dz$.

First, for probability distribution of the form $\mu_B$, one has the following useful identity for $F^\epsilon[\mu_B]$.
\begin{lemma}\label{lem:identityF}
For every $B\in\mathcal{C}(\T^m,\R)$, one has
\[
F^\epsilon[\mu_B]=\frac{\int \nabla A_\star(z)K_\epsilon(z,\cdot)e^{B(z)-A_\star(z)}dz}{\int K_\epsilon(z,\cdot)e^{B(z)-A_\star(z)}dz}.
\]
\end{lemma}

\begin{proof}
This is a straightforward consequence of the two identities below: for all $z\in\T^m$,
\begin{align*}
\int e^{-V(y,z)}dy&=e^{-A_\star(z)},\\
\int \nabla_zV(y,z)e^{-V(y,z)}dy&=-\nabla\left(\int e^{-V(y,z)}dy\right)=e^{-A_\star(z)}\nabla A_\star(z).
\end{align*}
\end{proof}

The set of the probability distribution of the type $\mu_B$ is an attractor for the dynamics of the limit flow, more precisely one has the following result.
\begin{propo}\label{propo:attractor}
One has the following result: for all $t\ge 0$,
\[
\underset{\epsilon>0}\sup~\underset{\mu\in\mathcal{P}(\T^d)}\sup~ \underset{B\in\mathcal{C}(\T^m,\R)}\inf~d_{\rm TV}(\Phi^\epsilon(t,\mu),\mu_B)\le 2e^{-t}.
\]
\end{propo}

\begin{proof}
For all $t\ge 0$ and $\mu\in\mathcal{P}(\T^d)$, one has
\[
\Phi^\epsilon(t,\mu)=e^{t}\mu+\int_{0}^{t}e^{s-t}\Pi^\epsilon[\Phi^\epsilon(s,\mu)]ds=e^{-t}\mu+(1-e^{-t})\Psi^{\epsilon}(t,\mu),
\]
where $\Psi^\epsilon(t,\mu)=\frac{1}{1-e^{-t}}\int_{0}^{t}e^{s-t}\Pi^\epsilon[\Phi^\epsilon(s,\mu)]ds=\mu_B$ for some $B\in \mathcal{C}(\T^m,\R)$, owing to the definition of $\Pi^\epsilon$.

Then
\[
\underset{B\in\mathcal{C}(\T^m,\R)}\inf~d_{\rm TV}(\Phi^\epsilon(t,\mu),\mu_B)\le d_{\rm TV}(\Phi^\epsilon(t,\mu),\Psi^\epsilon(t,\mu))\le e^{-t}\|\mu-\Psi^\epsilon(t,\mu)\|_{\rm TV}\le 2e^{-t}.
\]
\end{proof}

\begin{lemma}\label{lem:distA}
For every $p\in[2,\infty)$, there exists $C_p\in(0,\infty)$, such that for every $\epsilon>0$, and every $B\in \mathcal{C}(\T^m,\R)$, one has
\begin{equation}
\|A^\epsilon[\mu_B]-\bar A_\star\|_{W^{1,p}(\T^m)}\le C_p\sqrt{\epsilon}e^{2\bigl(\|B\|_{\mathcal{C}^0}+\|A_\star\|_{\mathcal{C}^0}\bigr)}.
\end{equation}
\end{lemma}

Recall that $\bar A_\star=A_\star - \int_{\T^m} A_\star dz$.

\begin{proof}
Using Proposition~\ref{propo:PDE}, one has the following inequality:
\[
\|A^\epsilon[\mu_B]-\bar A_\star\|_{W^{1,p}(\T^m,\R)}\le C_p\|F^\epsilon[\mu_B]-\nabla A_\star\|_{C^0(\T^m,\R^m)}.
\]
Owing to Lemma~\ref{lem:identityF} and using the Lipschitz continuity of $A_\star$, for all $z\in\T^m$, one has
\begin{align*}
\big|F^\epsilon[\mu_B](z)-\nabla A_\star(z)\big|&\le \Big|\frac{\int \bigl(\nabla A_\star(z')-\nabla A_\star(z)\bigr)K_\epsilon(z',z)e^{B(z')-A_\star(z')}dz'}{\int K_\epsilon(z',z)e^{B(z')-A_\star(z')}dz'}\Big|\\
&\le C\frac{\int |z-z'|K_\epsilon(z',z)dz' e^{\|B\|_{\mathcal{C}^0}+\|A_\star\|_{\mathcal{C}^0}}}{\int K_\epsilon(z',z)dz' e^{-\|B\|_{\mathcal{C}^0}-\|A_\star\|_{\mathcal{C}^0}}}\\
&\le C\sqrt{\epsilon}e^{2(\|B\|_{\mathcal{C}^0}+\|A_\star\|_{\mathcal{C}^0})},
\end{align*}
owing to Assumption~\ref{ass:kernel}. This inequality concludes the proof.
\end{proof}

\subsection{Contraction property on the attracting set}

Let $M\in(0,\infty)$. Introduce the set
\[
\mathcal{B}_M=\left\{B\in\mathcal{C}^0(\T^m,\R),~\int B(z)dz=0,~\|B\|_{\mathcal{C}^0}\le M \right\}.
\]
Owing to Proposition~\ref{propo:bound}, if $M\ge M^0$, then $A^\epsilon[\mu]\in\mathcal{B}_M$ for every $\mu\in\mathcal{P}(\T^d)$ and $\epsilon>0$.

Introduce the notation
\[
h_B(y,z)=\mathcal{Z}_B^{-1}e^{-V(y,z)+B(z)}\qquad\text{and}\qquad\tilde{\Pi}^\epsilon[h_B]=h_{A^{\epsilon}[\mu_B]}\,,
\]
so that $h_B$ and $\tilde{\Pi}^\epsilon[h_B]$ are the density with respect to the lebesgue measure of, respectively, $\mu_B$ and $\Pi^\epsilon[\mu_B]$.

To state the following result, the notation $\|h\|_2=\bigl(\int h(x)^2dx\bigr)^{\frac12}$ is used.
\begin{propo}\label{propo:contractionPitilde}
For every $M\in(0,\infty)$, there exists $C_M\in(0,\infty)$, such that for all $\epsilon>0$ and all $B^1,B^2\in\mathcal{B}_M$, one has
\[
\|\tilde{\Pi}^\epsilon[h_{B^1}]-\tilde{\Pi}^\epsilon[h_{B^2}]\|_2\le C_M\sqrt{\epsilon}\|h_{B^1}-h_{B^2}\|_2.
\]
\end{propo}

\begin{proof}
Let $B^1,B^2\in\mathcal{B}_M$. Using Proposition~\ref{propo:bound}, one has
\[
\|\tilde{\Pi}^\epsilon[h_{B^1}]-\tilde{\Pi}^\epsilon[h_{B^2}]\|_2=\|h_{A^{\epsilon}[\mu_{B^1}]}-h_{A^{\epsilon}[\mu_{B^2}]}\|_2\le C\|A^\epsilon[\mu_{B^1}]-A^\epsilon[\mu_{B^2}]\|_2.
\]
In addition, using the Poincar\'e inequality and the definition of $A^\epsilon[\mu]$ as the orthogonal projection in $L^2$ of $F^\epsilon[\mu]$, one has
\[
\|A^\epsilon[\mu_{B^1}]-A^\epsilon[\mu_{B^2}]\|_2\le C\|F^\epsilon[\mu_{B^1}]-F^\epsilon[\mu_{B^2}]\|_2.
\]
Then, using Lemma~\ref{lem:identityF}, one obtains, for all $z\in\T^m$,
\begin{align*}
|F^\epsilon[\mu_{B^1}](z)-&F^\epsilon[\mu_{B^2}](z)|=
\Big|\frac{\int \bigl(\nabla A_\star(z')-\nabla A_\star(z)\bigr)K_\epsilon(z',z)e^{B^1(z')-A_\star(z')}dz'}{\int K_\epsilon(z',z)e^{B^1(z')-A_\star(z')}dz'}\\
&\quad\quad\quad-\frac{\int \bigl(\nabla A_\star(z')-\nabla A_\star(z)\bigr)K_\epsilon(z',z)e^{B^2(z')-A_\star(z')}dz'}{\int K_\epsilon(z',z)e^{B^2(z')-A_\star(z')}dz'}\Big|\\
&\le \Big|\frac{\int \bigl(\nabla A_\star(z')-\nabla A_\star(z)\bigr)K_\epsilon(z',z)\Bigl(e^{B^1(z')}-e^{B^2(z')}\Bigr)e^{-A_\star(z')}dz'}{\int K_\epsilon(z',z)e^{B^1(z')-A_\star(z')}dz'}\Big|\\
&~+\Big|\frac{\int \bigl(\nabla A_\star(z')-\nabla A_\star(z)\bigr)K_\epsilon(z',z)e^{B^2(z')-A_\star(z')}dz' \int K_\epsilon(z',z)\Bigl(e^{B^1(z')}-e^{B^2(z')}\Bigr)e^{-A_\star(z')}dz'}{\int K_\epsilon(z',z)e^{B^1(z')-A_\star(z')}dz'\int K_\epsilon(z',z)e^{B^2(z')-A_\star(z')}dz'}\Big|\\
&\le Ce^{\|B^1\|_{\mathcal{C}^0(\T,\R)}}\int |z'-z|K_\epsilon(z',z)|e^{B^1(z')}-e^{B^2(z')}|dz'\\
&~+Ce^{\|B^1\|_{\mathcal{C}^0(\T,\R)}+2\|B^2\|_{\mathcal{C}^0(\T,\R)}}\int |z'-z|K_\epsilon(z',z)dz' \int K_\epsilon(z',z)|e^{B^1(z')}-e^{B^2(z')}|dz',
\end{align*}
using Lipschitz continuity of $\nabla A_\star$, and the lower bound
\[
\int K_\epsilon(z',z)e^{B^i(z')-A_\star(z')}dz'\ge e^{-\|B^i\|_{\mathcal{C}^0(\T,\R)}-\|A_\star\|_{\mathcal{C}^0(\T)}}\int K_\epsilon(z',z)dz'=e^{-\|B^i\|_{\mathcal{C}^0(\T,\R)}-\|A_\star\|_{\mathcal{C}^0(\T)}}.
\]
One has $\|B^1\|_{\mathcal{C}^0}\le M$ and $\|B^2\|_{\mathcal{C}^0}\le M$. In addition, owing to Assumption~\ref{ass:kernel}, one has $\int |z'-z|K_\epsilon(z',z)dz'\le C\sqrt{\epsilon}$. As a consequence, using the Jensen inequality (since $\int K_\epsilon(z',z)dz'=\int K_\epsilon(z,z')dz'=1$ for all $z$), one obtains
\begin{align*}
\|F^\epsilon[\mu_{B^1}]-F^\epsilon[\mu_{B^2}]\|_2&\le C_M\iint K_\epsilon(z',z)|z'-z|^2|e^{B_1(z')}-e^{B_2}(z')|^2dzdz'\\
&\quad+C_M\epsilon \iint K_\epsilon(z',z)|e^{B_1(z')}-e^{B_2}(z')|^2dzdz'\\
&\le C_M\epsilon \int|e^{B_1(z')}-e^{B_2(z')}|^2dz'.
\end{align*}
It remains to check that 
\[
\int|e^{B_1(z')}-e^{B_2}(z')|^2dz'\le C\|h_{B_1}-h_{B_2}\|_2^2.
\]
On the one hand,
\begin{align*}
\|h_{B_1}-h_{B_2}\|_2^2&=\iint e^{-2V(y,z)}\Big|\frac{e^{B_1(z)}}{\int e^{B_1-A_\star}}-\frac{e^{B_2}(z)}{\int e^{B_2-A_\star}}\Big|^2dydz\\
&\ge c\int \Big|\frac{e^{B_1(z)}}{\int e^{B_1-A_\star}}-\frac{e^{B_2}(z)}{\int e^{B_2-A_\star}}\Big|^2dz,
\end{align*}
with $c>0$. On the other hand, using Young inequality (with auxiliary parameter $\eta>0$), one obtains
\begin{align*}
\int|e^{B_1(z')}-e^{B_2}(z')|^2dz'&=\Big|\int e^{B_1-A_\star}\frac{e^{B_1(z)}}{\int e^{B_1-A_\star}}-\int e^{B_2-A_\star}\frac{e^{B_2}(z)}{\int e^{B_2-A_\star}}\Big|^2dz\\
&\le 2\eta^2\int \Big|\frac{e^{B_2}(z)}{\int e^{B_2-A_\star}}\Big|^2dz \Big|\int e^{B_1-A_\star}-\int e^{B_2-A_\star}\Big|^2\\
&~+\frac{2}{\eta^2}\bigl(\int e^{B_1-A_\star}\bigr)^2 \int \Big|\frac{e^{B_1(z)}}{\int e^{B_1-A_\star}}-\frac{e^{B_2}(z)}{\int e^{B_2-A_\star}}\Big|^2dz\\
&\le 2C_M\eta^2\int|e^{B_1(z')}-e^{B_2}(z')|^2dz'+\frac{2C_M}{\eta^2}\|h_{B_1}-h_{B_2}\|_2^2.
\end{align*}
Choosing a sufficiently small parameter $\eta$ one finally obtains the claim above.

Gathering the estimates finally concludes the proof of the estimate
\[
\|\tilde{\Pi}^\epsilon[h_{B^1}]-\tilde{\Pi}^\epsilon[h_{B^2}]\|_2\le C_M\sqrt{\epsilon}\|h_{B^1}-h_{B^2}\|_2.
\]
\end{proof}

\subsection{Proof of the main result}

The first part of this section is devoted to the construction of the candidate limits $\mu_\infty^\epsilon$ and $A_\infty^\epsilon=A^\epsilon[\mu_\infty^\epsilon]$, of $\mu_t$ and $A_t$ respectively, for small enough $\epsilon$.

Let $\bar \epsilon_0=1/(C_{M^{(0)}}^2+1)$, where $M=M^{(0)}$ is given by Proposition~\ref{propo:bound} and $C_M$ is given by Proposition~\ref{propo:contractionPitilde}.

Let $\epsilon\in(0,\bar \epsilon_0]$, and consider $A_{(0)}\in \mathcal{B}_{M^{(0)}}$. Define $\mu_{(0)}=\mu_{A_{(0)}}$, $h_{(0)}=h_{A_{(0)}}$, and by recursion, for all nonnegative integer $k$, let
\[
\mu_{(k+1)}=\Pi^\epsilon[\mu_{(k)}]~,\quad h_{(k+1)}=\tilde{\Pi}^{\epsilon}[h_{(k)}],
\]
and let $A_{(k)}=A^\epsilon[\mu_{(k)}]$. Then one has $h_{(k)}=h_{A_{(k)}}\in \mathcal{B}_{M^{(0)}}$. We claim that $\bigl(\mu_{(k)}\bigr)_{k\ge 0}$ is a Cauchy sequence in the space $\mathcal{P}(\T^d)$ equipped with the total variation distance $d_{\rm TV}$. Indeed, for all $k,\ell\ge 0$, one has
\begin{align*}
d_{\rm TV}(\mu_{(k)},\mu_{(k+\ell)})& \leqslant \|h_{(k)}-h_{(k+\ell)}\|_2\\
&\le \bigl(C_{M^{(0)}}\sqrt{\epsilon}\bigr)^kd_2(h^{(0)},h^{(\ell)})\\
&\le C\rho^k,
\end{align*}
with $\rho\in(0,1)$. As a consequence, there exists $\mu_\infty^\epsilon$ such that $d_{\rm TV}(\mu_{(k)},\mu_\infty^\epsilon)\underset{k\to\infty}\to 0$. Owing to Lemma~\ref{lem:Lip_PiEpsilon}, the mapping $\Pi^\epsilon$ is continuous on $\mathcal{P}(\T^d)$ equipped with $d_{\rm TV}$, thus $\mu_\infty^\epsilon=\Pi^\epsilon[\mu_\infty^\epsilon]$. This implies that $\mu_\infty^\epsilon=h_{A_\star^\epsilon}(x)dx$ where $A_\star^\epsilon=A^\epsilon[\mu_\infty^\epsilon]\in\mathcal{B}_{M^{(0)}}$.

It is then straightforward to check that $h_\infty^\epsilon=h_{A_\infty^\epsilon}$ is the unique fixed point of the mapping $\tilde{\Pi}^\epsilon$ (uniqueness is a consequence of Proposition~\ref{propo:contractionPitilde}).

We claim that, for any initial condition of the type $\mu_B$, then $\Phi^\epsilon(t,\mu_B)\underset{t\to\infty}\to \mu_\infty^\epsilon$, more precisely one has exponential convergence to the fixed point $\mu_\infty^\epsilon$: there exists $c(\epsilon)\in(0,\infty)$ such that, for all $t\ge 0$, one has
\begin{equation}\label{eq:claim_lastproof}
\underset{B\in\mathcal{B}_{M^{(0)}}}\sup~d_{\rm TV}(\Phi^\epsilon(t,\mu_B),\mu_\infty^\epsilon)\le Ce^{-c(\epsilon)t}.
\end{equation}
To prove this claim, observe that for all $t\ge 0$, the probability distribution $\Phi^\epsilon(t,\mu_B)$ can be written as $\mu_{B_t}$, where $B_t\in\mathcal{C}^0$, see Proposition~\ref{propo:attractor}, and without loss of generality $\int B_t(z)dz=0$. In addition, $B_t\in\mathcal{B}_{M^{(1)}}$, for all $t\ge 0$, for some $M^{(1)}\in(0,\infty)$ depending only on $M^{(0)}$: indeed, the identity
\[
h_{B_t}=e^{-t}h_{B_0}+\int_{0}^{t}e^{-(t-s)}\tilde{\Pi}^\epsilon[h_{B_s}]ds
\]
implies, using Proposition~\ref{propo:bound}, the bounds
\[
0<\underset{t\ge 0}\inf~\underset{x\in \T^d}\inf~h_{B_t}(x)\le \underset{t\ge 0}\sup~\underset{x\in \T^d}\sup~h_{B_t}(x)<\infty,
\]
and $B_t(z)$ is equal (up to an additive constant defined to respect the condition $\int B_t(z)dz=0$) to $A_\star(z)+\log\bigl(\int e^{-V(y,z)}dy)$.

Let $\epsilon_0=1/(C_{M^{(1)}}^2+1)$, and assume in the sequel that $\epsilon\in(0,\epsilon_0]$. Note that $M^{(1)}\ge M^{(0)}$, thus $\epsilon_0\le \bar \epsilon_0$.

Then $A_\infty^\epsilon$ is well-defined, $h_\infty^\epsilon$ is the unique fixed point of $\tilde{\Pi}^\epsilon$, and one obtains
\begin{align*}
d_{\rm TV}(\Phi^\epsilon(t,\mu_B),\mu_\infty^\epsilon)&\le \|h_{B_t}-h_{\infty}^\epsilon\|_2\\
&\le e^{-t}\|h_{B}-h_\infty^\epsilon\|_2+\int_{0}^{t}e^{-(t-s)}\|\tilde{\Pi}^\epsilon[h_{B_s}]-\tilde{\Pi}^{\epsilon}[h_\infty^\epsilon]\|_2ds\\
&\le e^{-t}\|h_{B}-h_\infty^\epsilon\|_2+C_{M^{(1)}}\sqrt{\epsilon}\|h_{B_s}-h_{\infty}^\epsilon\|_2 ds,
\end{align*}
with $C_{M^{(1)}}\sqrt{\epsilon}<1$. Applying the Gronwall Lemma, one obtains
\[
d_{\rm TV}(\Phi^\epsilon(t,\mu_B),\mu_\infty^\epsilon)\le \|h_{B_t}-h_{\infty}^\epsilon\|_2\le e^{-(1-C_{M^{(1)}}\sqrt{\epsilon})t}\|h_{B}-h_\infty^\epsilon\|_2,
\]
and it is straightforward to check that $ \sup~\{\|h_{B}-h_\infty^\epsilon\|_2,\ B\in\mathcal{B}_{M^{(0)}}\}<\infty$. This concludes the proof of the claim~\eqref{eq:claim_lastproof}.

We are now in position to prove give the proof of Theorem~\ref{theo:Main}. It is sufficient to focus on the question of convergence when $t\to\infty$, indeed the estimate for $\|A_\infty^\epsilon-\nabla A_\star\|_{W^{1,p}}$ is a straightforward consequence of Lemma~\ref{lem:distA}, combined with Proposition~\ref{propo:bound}, since $A_\infty^\epsilon$ is a fixed point of the mapping $A\mapsto A^\epsilon[\mu_A]$.

The idea of the proof, using concepts and tools developed in~\cite{B99} may be described as follows. Since almost surely $\bigl(\mu_t\bigr)_{t\ge 0}$ is an asymptotic pseudo-trajectory for the semi-flow $\Phi^\epsilon$, one has the following property: the limit set $L(\mu)$ is an attractor free set for the semi-flow $\Phi^\epsilon$ in $\mathcal{P}(\T^d)$, in particular it is invariant, {\it i.e.} for all $t\ge 0$ one has $\Phi^\epsilon(t,L(\mu))=L(\mu)$. Let us check that $L(\mu)=\{\mu_\infty^\epsilon\}$. First, introduce the set $\mathcal{M}=\left\{\mu_B\right\}$. Then Proposition~\ref{propo:attractor} provides the inclusion $L(\mu)\subset \mathcal{M}$. Indeed, let $\nu\in L(\mu)$ and let $t\ge 0$ be arbitrary, then by invariance there exists $\tilde{\nu}\in L(\mu)$ such that $\nu=\Phi^{\epsilon}(\tilde{\nu})$, thus $d(\nu,\mathcal{M})=d(\Phi^\epsilon(\tilde{\nu}),\mathcal{M})\le 2e^{-t}\underset{t\to\infty}\to 0$. Similarly, let $\nu\in L(\mu)\subset\mathcal{M}$, and let $t\ge 0$ be arbitrary, then $\nu=\Phi^{\epsilon}(\tilde{\nu})$ for some $\tilde{\nu}\in \mathcal{M}$. Thus $d(\nu,\mu_\infty^\epsilon)=d(\Phi^\epsilon(t,\tilde{\nu}),\Phi^\epsilon(t,\mu_\infty^\epsilon))\le Ce^{-ct}\underset{t\to\infty}\to 0$.

Let us now provide a detailed proof using only the results presented above.

\begin{proof}[Proof of Theorem~\ref{theo:Main}]

Let $T_1,T_2\in(0,\infty)$ be arbitrary positive real numbers, and $T=T_1+T_2$. For every $t\ge T$, one has
\begin{align*}
d_w\bigl(\mu_{e^t},\mu_\infty^\epsilon\bigr)&\le d_w\bigl(\mu_{e^t},\Phi^\epsilon(T,\mu_{e^{t-T}}\bigr)+d_w\bigl(\Phi^\epsilon(T,\mu_{e^{t-T}}),\mu_\infty^\epsilon\bigr).
\end{align*}
Owing to Theorem~\ref{theo:APT}, for any fixed $T_1,T_2$, one has, almost surely,
\[
d_w\bigl(\mu_{e^t},\Phi^\epsilon(T,\mu_{e^{t-T}})\bigr)\underset{t\to\infty}\to 0.
\]
Observe that $d_w(\cdot,\cdot)\le Cd_{\rm TV}(\cdot,\cdot)$. In addition, for all $B\in\mathcal{C}(\T,\R)$, using Lemma~\ref{lem:Lip_PiEpsilon} and the claim~\eqref{eq:claim_lastproof} above, one has
\begin{align*}
d_{\rm TV}\bigl(\Phi^\epsilon(T,\mu_{e^{t-T}}),\mu_\infty^\epsilon\bigr)&\le d_{\rm TV}\bigl(\Phi(T_1,\Phi(T_2,\mu_{e^{t-T}})),\Phi(T_1,\mu_{B})\bigr)+d_{\rm TV}\bigl(\Phi(T_1,\mu_B),\mu_\infty^\epsilon\bigr)\\
&\le e^{L(\epsilon)T_1}d_{\rm TV}\bigl(\Phi(T_2,\mu_{e^{t-T}}),\mu_B\bigr)+Ce^{-c(\epsilon)T_1}
\end{align*}
This implies that
\begin{align*}
d_{\rm TV}\bigl(\Phi^\epsilon(T,\mu_{e^{t-T}}),\mu_\infty^\epsilon\bigr)&\le e^{L(\epsilon)T_1}\underset{B\in\mathcal{C}(\T,\R)}\sup~d_{TV}(\Phi(T_2,\mu_{e^{t-T}}),\mu_B)+Ce^{-c(\epsilon)T_1}\\
&\le 2e^{L(\epsilon)T_1}e^{-T_2}+2e^{-c(\epsilon)T_1},
\end{align*}
owing to Proposition~\ref{propo:attractor}.
\[
\underset{t\to \infty}\limsup~d_{\rm TV}\bigl(\Phi^\epsilon(T,\mu_{e^{t-T}}),\mu_\star^\epsilon\bigr)\le 2e^{L(\epsilon)T_1}e^{-T_2}+2e^{-c(\epsilon)T_1}.
\]
Since $T_1$ and $T_2$ are arbitrary, letting first $T_2\to\infty$, then $T_1\to\infty$, one has almost surely
\[
\underset{t\to\infty}\limsup~d_w(\mu_{e^t},\mu_\star^\epsilon)=0,
\]
which concludes the proof of the weak convergence of $\mu_t$ to $\mu_\infty^\epsilon$.

It remains to check that $A_t=A^\epsilon[\mu_t]$ converges to $A_\infty^\epsilon=A^\epsilon[\mu_\infty^\epsilon]$, in $\mathcal{C}^k$, for all $k\in\N$. This is a consequence of the regularity properties of $K_\epsilon$ and of $V$, which proves that $\mu\in (\mathcal{P}(\T^d),d_w)\mapsto F^\epsilon[\mu]\in \mathcal{C}^k$ is continuous for all $k\in\N$.

Using Sobolev embedding properties, as in the proof of Lemma~\ref{lem:bounds-F_A}, then concludes the proof.
\end{proof}

\section*{Acknowledgments}

The authors would like to thank Pierre-Damien Thizy for pointing out the relevant reference~\cite{Ambrosio}, for the PDE estimate presented in Proposition~\ref{propo:PDE}. This work has been partially supported by the Project EFI ANR-17-CE40-0030 of the French National Research Agency and by the SNF grant 200021 - 175728.

\bibliographystyle{plain}

\end{document}